\documentclass[12pt,reqno]{amsart}
\usepackage{amsmath,amssymb,amsfonts,amscd,latexsym,amsthm,mathrsfs,verbatim,textcomp}
\usepackage[colorlinks,linkcolor=black,citecolor=black]{hyperref}
\usepackage{cite}
\usepackage{hypbmsec}
\usepackage{enumerate}
\usepackage{bm}
\usepackage{tikz}
\usepackage{mathtools}
\usepackage{verbatim}
\usepackage{xfrac, mathabx}
\theoremstyle{plain}
\usepackage{manfnt}
\usepackage{tikz}
\usetikzlibrary{matrix,arrows,decorations.pathmorphing}
\usepackage[margin=1.1in]{geometry}
\usepackage{csquotes}

\usepackage{comment}
\newtheorem{lemma}{Lemma}

\newtheorem{theorem}{Theorem}
\newtheorem{conjecture}{Conjecture}
\usepackage{amsthm}
\usepackage{amssymb}

\numberwithin{equation}{section}
\numberwithin{figure}{section}
\theoremstyle{plain}
\newtheorem{thm}{\protect\theoremname}

\newtheorem{prop}{\protect\propositionname}
\theoremstyle{definition}

\theoremstyle{remark}
\newtheorem{rem}[thm]{\protect\remarkname}

\makeatletter

\newcommand{\mmod}{\,\mathrm{mod}\,}

\makeatother

\usepackage{babel}
\providecommand{\definitionname}{Definition}
\providecommand{\remarkname}{Remark}
\providecommand{\theoremname}{Theorem}
\providecommand{\corollaryname}{Corollary}
\providecommand{\propositionname}{Proposition}

\makeatletter
\@namedef{subjclassname@2020}{%
  \textup{2020} Mathematics Subject Classification}
\makeatother

\address{Lunt Hall, Room 201, Northwestern University,
2033 Sheridan Road, Evanston, IL, 60208, USA}
 \email{maksym.radziwill@gmail.com}
 
\address{Former Address: Max Planck Institute for Mathematics, 
Vivatsgasse 7, 53111~Bonn, Germany;
Current Address: 
Department of Mathematics, 
University of Wisconsin 480 Lincoln Drive, Madison, WI, 53706, USA.}
\email{technau@wisc.edu}
  
\subjclass[2020]{11J71; 37A44; 65C20}
\keywords{gap distribution; uniform distribution;
pseudo-randomness}

\begin{document}

\title{Gap distribution of $\sqrt{n} \pmod {1}$ and the circle method}
 \author{Maksym Radziwi\l{}\l{} and Niclas Technau}

 \dedicatory{ 
 Dedicated to Andrew Granville on the occasion of his $61^{th}$
 birthday}

 \begin{abstract} 
  The distribution of the properly 
  renormalized gaps of 
  $\sqrt{n} \mmod{1}$ with $n < N$
  converges (when $N\rightarrow \infty$)
  to a non-standard limit distribution, 
  as Elkies and McMullen proved in 2004
  using techniques from homogeneous dynamics. 
  In this paper we give an essentially 
  self-contained proof based on the circle method. 
  Our main innovation
  consists in showing that a new
  type of correlation functions of 
  $\sqrt{n} \mmod{1}$
  converges. To define these 
  correlation functions we restrict,
  smoothly, to those 
  $\sqrt{n} \mmod{1}$ that lie in
  minor arcs, i.e. away from rational numbers with
  small denominators. %To relate
 % the new and the classical correlation functions,
 % we analyze the pair correlation function of
 % rough and square-free Farey fractions, which might be of independent interest.%Studying such
  %correlations suffices to compute a limiting
  %gap distribution 
  %and is hence of wider interest.
\end{abstract}
 
 \maketitle

\section{Introduction}

By work of Fej\'{e}r and of Csillag,
it is known that the sequence $\{n^{\alpha}\}_n$, with $\alpha > 0$ non-integer, is uniformly distributed modulo $1$. 
A much finer question asks for the \textit{gap distribution} of 
$\{ n^{\alpha} \mmod{1}\}_n$; that is after ordering the values 
$n^{\alpha} \mmod{1}$ with $n \in [N, 2N)$ in increasing order we ask for the gap distribution (at the scale $1/N$) between the ordered values. A striking folklore conjecture in this area is the following.
  \begin{conjecture}
    Let $\alpha > 0$ be non-integer. 
    Then the gap distribution of 
    $\{ n^{\alpha} \mmod{1}\}_n$ exists. 
    Furthermore the gap distribution is 
    exponential if and only if $\alpha \neq \tfrac 12$. 
  \end{conjecture}
The only case in which this conjecture is known is for $\alpha = \tfrac 12$, a result of Elkies-McMullen \cite{EM}. For the remaining exponents $\alpha$ it is confirmed that the pair correlation is Poissonian when $\alpha < \tfrac {43}{117}$ 
(see the work
of Shubin and 
the first named author \cite{radziwill2023poissonian}
and previous work of Sourmelidis, Lutsko
and the second named author \cite{A}),
and that higher correlations are Poissonian when $\alpha$ 
shrinks to zero (see \cite{B,C}) or
when $\alpha$ is randomized
and large (see \cite{TY}). 
Coincidentally the pair correlation of 
$\{ \sqrt{n} \mmod 1 \}_n$ is also Poissonian 
after removal of squares \cite{Pair}; 
however higher correlations do not exist 
even after removal of the squares. 
We refer to \cite{Pattison} for related pigeon-hole statistics.

  The proof of Elkies and McMullen relies on a connection with ergodic theory; their proof crucially uses a special case of Ratner's theorem (see \cite{Effective} for an effective version in this special case). In this paper we aim to give a new proof of the Elkies-McMullen theorem based on the circle method.
  %This conjecture appeared recently in a more specific form,
  %asking $\alpha^2$ to be Diophantine,
  %in Zeev Rudnick's talk in Number Theory Web Seminar
  %(see the last page of \cite{Rudnick_Talk}).
   The main disadvantage of our method is that identifying the limit explicitly requires an additional effort (that we do not undertake here); however implicit in the proof is a way to rigorously evaluate the limit to any required precision. In particular one could verify with a computer calculation that the limit is indeed not exponential.

\begin{theorem} \label{thm:elkiesmcmullen}
  Let $N \geq 1$.
  Let $\delta$ be the Dirac delta function 
  at the origin. Let $\{s_{n, N}\}_{N \leq n \leq 2N}$ 
  denote the (numerical) ordering of 
  $\{ \sqrt{n} \mmod 1 \}_{N \leq n \leq 2N}$. Then, the measure
  $$
  \frac{1}{N} \sum_{N \leq n < 2N} 
  \delta ( N (s_{n + 1, N} - s_{n,N}))
  $$
  converges weakly to a limit distribution $P(u)$ as $N \rightarrow \infty$, in other words, for every smooth compactly supported $f$, 
  $$
 \lim_{N \rightarrow \infty} \frac{1}{N} \sum_{N \leq n < 2N} 
 f ( N (s_{n + 1, N} - s_{n, N}) ) 
 = \int_{\mathbb{R}} f(u) P(u) du .
  $$
\end{theorem}
Elkies and McMullen's original result considers gaps with $n < N$, we choose to consider the case $N \leq n < 2N$ because it avoids some superficial technicalities.
There are some apparently new results that one can obtain with our method, for example we can show the existence
of an explicit irrational $\alpha$ such that the distribution of $\alpha \sqrt{n} \mmod{1}$ is Poisson along a subsequence. This partially addresses a question left open by Elkies and McMullen, who asked for the existence of an irrational $\alpha$ such that the distribution of $\alpha \sqrt{n} \mmod{1}$ is Poisson along \textit{all} subsequences.  We plan to investigate any further possible applications at a later time. %in a subsequent paper. 

\subsection{Notations and conventions}
Throughout for a Schwartz function $f$ we define
$$
\widehat{f}(\xi) := \int_{\mathbb{R}} f(x) e( - \xi x) dx
$$
and we set $e(x) := e^{2\pi i x}$. 
We also write $\| x \|$ to denote the reduction of $x$ modulo $1$, i.e $\|x \| := x \mmod{1} \in [0,1)$.
Given two integers $a,q\geq 1$,
we denote by $(a,q)$ and $[a,q]$ the greatest common
divisor and the least common multiple, respectively. 
We will frequently abbreviate the summation constraints 
$1\leq a \leq q$ and $ (a,q)=1$ by simply writing $(a,q)=1$,
for example
$$
\sum_{\substack{1\leq a \leq q \\ (a,q)=1}} a 
=\sum_{(a,q)=1} a.
$$

\subsection{Outline of proof}

By a standard reduction (that we recall in the appendix), 
to prove that the gap distribution exists, 
it suffices to show that the so-called 
\textit{void statistic} $\mathcal{V}(\cdot,N)$ 
at time $N$ has a limiting behavior (as $N\rightarrow \infty$) 
where
$$
\mathcal{V}(s,N)
:= \int_{0}^{1} \mathbf{1} 
  (\mathcal{N}(x; s, N) = 0 ) dx
$$
 and $\mathcal{N}(x; s, N)$ counts the number of 
  $ \sqrt{n} \mmod 1$ in $[x, x + s / N)$, with $N \leq n < 2N$.

   \begin{prop} \label{prop0}
     If for every $s > 0$ the limit
     $ \lim_{N \rightarrow \infty} \mathcal{V}(s,N)$
     exists then Theorem \ref{thm:elkiesmcmullen} holds. 
\end{prop}

The main observation driving our approach is that $\mathcal{N}(x; s, N)$ can be efficiently understood when $x$ belongs to the so-called \textit{minor arcs}
$$
\bigcup_{\substack{(a,q) = 1 \\ q \asymp \Delta \sqrt{N}}} 
\Big ( \frac{a}{q} - \frac{1}{N}, \frac{a}{q} + \frac{1}{N} \Big )
$$
with $\Delta$ a large fixed constant. The measure of this set converges to $1$ as $\Delta \rightarrow \infty$. 
Our first move is to effectuate a change of measure,
replacing Lebesgue measure by a measure $\mu_{N, \mathcal{Q}_{\Delta, N}}$ supported on minor arcs.
Given a smooth non-negative function $\phi:\mathbb{R}\rightarrow \mathbb{R}$ compactly supported in $[- \tfrac{1}{100}, \tfrac{1}{100}]$ and
such that
$\int_\mathbb{R} \phi(x)dx = 1$, 
we choose
\begin{equation}\label{def Jutila measure}
d\mu_{ N, \mathcal{Q}_{\Delta, N}}(x) := \frac{N}{\sum_{\substack{q \in \mathcal{Q}_{\Delta, N}}} \varphi(q)} \sum_{\substack{(a,q) = 1 \\ q \in \mathcal{Q}_{\Delta, N}}}  \phi \Big ( N \Big ( x - \frac{a}{q} \Big ) \Big ) dx
\end{equation}
and where,  
    \begin{align} \label{eq:qu}
    \mathcal{Q}_{\Delta, N} := 
    \Big \{ a b \in [Q, 2Q]: Q^{1/1000} \leq a \leq 2 Q^{1/1000}, p | a b \implies p > \Delta^{2001}, \mu^2(a b) = 1 \Big \}
    % \Bigg\{ a b \in [Q, 2Q] :
    % \begin{array}{l} 
    %  Q^{1/1000} \leq a \leq 2 Q^{1/1000}  \\
    %  \ p | a b \implies p > \Delta^{2001}  \\
    %   \mu^2(a b) = 1
    %  \end{array}
    %  \Bigg\},
    \end{align}
    with $Q := \Delta \sqrt{N}$. The function $\phi$ is considered globally fixed, for example $\phi(x) = c \exp( - (x - 1/100)^2) \mathbf{1}_{|x| \leq 1/100}$ with an appropriate constant $c$.   
    The bilinear structure of $\mathcal{Q}_{\Delta, N}$ will be useful later. The proposition below is a nearly immediate consequence of Jutila's version of the circle method \cite{Jutila} in the form from \cite{Mayank}.

    \begin{prop} \label{prop1}
      Let $N, \Delta > 10$
      be such that $\Delta^{2001} \leq \log N$. 
 The restricted void statistic
  $$
  \mathcal{V}(s,N; \mu_{N, \mathcal{Q}_{\Delta, N}})
   := \int_{0}^{1} 
  \mathbf{1} ( \mathcal{N}(x; s, N) = 0 ) 
  \, d\mu_{ N, \mathcal{Q}_{\Delta, N}}(x)
  $$
  approximates the void statistic
  with the following quality:
  $$
  | 
  %\int_{0}^{1} 
  %\mathbf{1} ( \mathcal{N}(x; s, N) = 0 ) \, dx 
  \mathcal{V}(s,N)
  - 
  %\int_{0}^{1} \mathbf{1} ( \mathcal{N}(x; s, N) 
  %= 0 ) 
 % \, d \mu_{\Delta, N}(x) 
 \mathcal{V}(s,N; \mu_{N, \mathcal{Q}_{\Delta, N}})
  | \ll \frac{1}{\Delta^{1/2}}
    $$
where the 
implied constant is absolute. 
\end{prop}
At the level of precision that we require it is easier to study a smoothed version of $\mathcal{N}(x; s, N)$. 
Hence, we consider
$$
R_{\Phi, V}(x; N) := \sum_{n \in \mathbb{Z}} V \Big ( \frac{n}{N} \Big ) \Phi ( N \| \sqrt{n} - x \| )
$$
where $\Phi,V$ are smooth functions
and $\Vert x \Vert := x \mmod 1$ is the reduction of $x$ modulo $1$. The next proposition shows that we can introduce such a smoothing.

\begin{prop} \label{prop:smoothingprop}
  Let $s > 0$, $\eta \in (0, \tfrac{1}{100})$ and $N > 100 / \eta$ be given.
  Take any smooth, 
  $0 \leq \Phi, V \leq 1$
  equal to one on respectively $(\eta,s)$ and $(1,2)$ and compactly supported in respectively $(0, s + \eta)$ and $(1 - \eta, 2 + \eta)$.
  Then, for any measure $\mu$, with $0 \leq d \mu \leq d x$, 
  $$
  |\mathcal{V}(s, N; \mu) - \mathcal{V}_{\Phi, V}(N; \mu)| \leq 8 (1 + s)\eta,
  $$
  where
  $$
  \mathcal{V}_{\Phi, V}(N ; \mu) := \int_{0}^{1} \mathbf{1} (R_{\Phi, V}(x; N) = 0) d \mu(x). 
  $$
\end{prop}

In the next proposition we show that the difference between the $\limsup$ and the $\liminf$ of the smoothed void statistic restricted to the measure $\mu_{ N, \mathcal{Q}_{\Delta, N}}(x)$ shrinks as $\Delta$ goes to infinity. 

\begin{prop} \label{prop2}
  Let $\eta \in (0, \tfrac{1}{100})$ and $\Delta, s > 0$ be given.   
  Let $0 \leq V, \Phi \leq 1$ be smooth functions, equal to one on 
  respectively $(1,2)$ and $(\eta, s)$ and compactly supported in respectively $(1 - \eta, 2 + \eta)$ and $(0, s + \eta)$. Then, 
  \begin{align*}
  \liminf_{N \rightarrow \infty} \mathcal{V}_{\Phi, V}(N; \mu_{N , \mathcal{Q}_{\Delta, N}})
  - \limsup_{N \rightarrow \infty} \mathcal{V}_{\Phi, V} 
  (N; \mu_{N, \mathcal{Q}_{\Delta, N}})  
  = O_{\Phi, V}(\Delta^{-10}) + O(\eta).
  \end{align*}
\end{prop}
Combining the above three propositions yields
$$
\liminf_{N \rightarrow \infty} \mathcal{V}(s, N) - \limsup_{N \rightarrow \infty} \mathcal{V}(s, N) = O_{\Phi, V}(\Delta^{-1/2}) + O((1 + s) \eta).
$$
Taking $\Delta \rightarrow \infty$ and then $\eta \rightarrow 0$ shows that the limit of $\mathcal{V}(s, N)$ exists for each $s > 0$, and thus Theorem \ref{thm:elkiesmcmullen} follows from Proposition \ref{prop1}. Notice that since $\Phi, V$ depend on $\eta$ we cannot take the limit in the opposite order ($\eta$ first, $\Delta$ later). We describe this deduction again and in more detail in Section \ref{se:proofall}.
Proposition \ref{prop2} is the core of the proof. The proof of Proposition \ref{prop2} relies on two results that we now describe. 

As hinted earlier, 
  $R_{\Phi, V}(x; N)$ admits a closed formula when $x$ lies in the ``minor arcs''. 
  
\begin{prop} \label{prop3}
  Let $N, \Delta > 10$ such that $\Delta^{2001} \leq \log N$ be given.  Let $A>10$.
  Let $0 \leq \Phi, V \leq 1$ be non-zero,
  smooth, compactly supported functions with $V$ compactly supported
 in $(1/2, 3)$. 
  Let $\theta \in \mathbb{R}$ be such that $|N \theta| \leq 1/100$.
   Then for any $(a,q) = 1$ and $q \in [\Delta \sqrt{N}, 2 \Delta \sqrt{N}]$
   such that all prime factors of $q$ are larger than $\Delta^{2001}$,
  \begin{align} \label{eq:form}
  R_{\Phi, V} \Big ( \frac{a}{q} - \theta; N \Big ) 
  = \widehat{\Phi}(0) \widehat{V}(0) & + 2 \widetilde{R}_{\Phi, V} 
  \Big ( \frac{a}{q} - \theta; N \Big ) + O_{A, \Phi, V} (\Delta^{-A}),
  \end{align}
  where
  $$
  \widetilde{R}_{\Phi, V} \Big ( \frac{a}{q} - \theta; N \Big ) 
  := \sum_{(u,v)\in \mathcal{B}_{a, \Delta, q}} 
  e \Big (  - \frac{\overline{q}^2 u^2}{4 v} + \frac{u^2}{4 v q^2} \Big ) 
  F_{\Phi, V, N \theta} 
  \Big ( \frac{v}{\sqrt{N}}, \frac{u}{q / \sqrt{N}} \Big ),
  $$
  the set $\mathcal{B}_{a, \Delta, q}$
  is defined by
  $$
  \mathcal{B}_{a, \Delta, q} :=
  \{(u,v)\in \mathbb{Z}^2: 
  1 \leq |u| \leq \Delta^{4},\,
  u \equiv 2 v a \mmod{q}
  \},
  $$
  and the weight function 
  $F_{\Phi, V, \theta}$ is given by
  $$
  F_{\Phi, V, \theta}(\xi, \eta) 
  := \int_{0}^{\infty} \widehat{\Phi}(2 \xi x) 
  e(2 \xi \theta x) x V(x^2) e \Big ( - \frac{\eta x}{2} \Big ) dx.
  $$
  Furthermore, we have the bound, 
  $$
  \widetilde{R}_{\Phi, V} \Big ( \frac{a}{q} - \theta; N \Big ) \ll_{\Phi, V} \Delta^4. 
  $$
\end{prop}
Some remarks are in order. 
\begin{rem}
  \phantom{a}
  \begin{enumerate}
\item[(a)] First, we claim that if
$(u,v)\in \mathcal{B}_{a, \Delta, q}$
then 
$v$ is coprime to $q$
and therefore the inverse $\overline{q}$ 
modulo $4v$ exists.
To see this notice that
if $d = (v,q)$  
then $d\mid u$ since $u \equiv 2va \mmod q $ and in particular $d \leq |u| \leq \Delta^4$. 
However since $q\in \mathcal{Q}_{\Delta,N}$ all prime 
factors of $q$ are larger than $\Delta^{2001}$, thus 
if $d > 1$ then $d \geq \Delta^{2001}$
contradicting 
$d \leq |u| \leq \Delta^{4}$. 
So, $(v,q)=1$. Trivially, we also have $u,v \neq 0$. 

\item[(b)] Furthermore, notice that the formula 
in Proposition \ref{prop3} 
gives a quick algorithm for determining $\mathcal{N}(x; s, N)$ for $x$ in the minor arcs. Trivially the sum contains $\ll \Delta^{O(1)}$ terms.
  \end{enumerate}
\end{rem}

In the next proposition we compute moments of $\tilde{R}_{\Phi, V}$ with respect to $d \mu_{ N, \mathcal{Q}_{\Delta, N}}(x)$. Without this restriction the moments would be divergent for high enough $k$. Instead thanks to the change of measure the $k$th moment is $\ll \Delta^{4 k}$. % restriction to $d \mu_{ N\Delta, \mathcal{Q}_{\Delta, N}}(x)$ the $k$th moment of $R_{\Phi, V}$ with $k > 3$ would be divergent. Instead when restricted to $d \mu_{ N\Delta, \mathcal{Q}_{\Delta, N}}(x)$ the $k$th moment of $\tilde{R}_{\Phi, V, \nu}$ grows no faster than $\Delta^{k - 3}$. This allows us to determine the limiting distribution of $\tilde{R}_{\Phi, V, \nu}$ using the moment method for each fixed $\Delta > 2$.
Finally, the bilinear structure of $\mathcal{Q}_{\Delta, N}$ is used here to obtain cancellations in certain exponential sums through a form of Weyl differencing.
%\textbf{TODO: Error term not correct, dependence on $\Phi$ not written down; main term also not correct because of incorrect integration by parts}
%\marginpar{Issue with $k = 1$ stuff doesn't make sense}
\begin{prop} \label{prop4}
  Let $0 \leq \Phi, V \leq 1$ be non-zero, smooth and compactly supported functions. 
  Assume that $V$ supported in $(1/2,3)$. Let $w$ be a smooth function, compactly supported in $(-1/2,1/2)$ with $w(0) = 1$. 
  Let $\Delta, N > 10$ be such that $\Delta^{2001} < \log N$. 
  Let $\phi$ be a smooth, non-negative function, compactly supported in $[- 1 /100, 1 /100]$ and normalized
  so that $\int_{\mathbb{R}} \phi(x) dx = 1$. 
  Write
  $$
  \boldsymbol{t} = (t_1, \ldots, t_{k}) \ , \ \boldsymbol{u} := (u_1, \ldots, u_k)
  $$
  and let $\widetilde{R}_{\Phi, V}(x; N)$ be as in Proposition \ref{prop3}. 
  Then, for every fixed $k \in \mathbb{N}$, 
  we have
  \begin{align*}
 & \frac{N}{\sum_{\substack{q \in \mathcal{Q}_{\Delta, N}}} \varphi(q)} \sum_{\substack{(a,q) = 1 \\ q \in \mathcal{Q}_{\Delta, N}}} \int_{\mathbb{R}} \widetilde{R}_{\Phi, V} \Big (\frac aq - \theta; N \Big )^k \phi(N \theta ) d \theta \\ & = \frac{2}{3 \Delta^2} \int_{\mathbb{R}} \int_{\Delta}^{2 \Delta} \sum_{\substack{\boldsymbol{u}, \boldsymbol{t} \in \mathbb{Z}^k \\ \langle \boldsymbol{1}, \boldsymbol{u} \rangle = 0 = \langle \boldsymbol{u}, \boldsymbol{t} \rangle \\ \forall i :  1 \leq |u_i| \leq \Delta^{4}}} H_{\Phi, V, k, \theta, \Delta} \Big ( \frac{\boldsymbol{t}}{v}, \frac{\boldsymbol{u}}{v} \Big ) \frac{dv}{v^{k - 1}} \phi(\theta) d \theta + O_{\Phi, V, k, \varepsilon}(N^{-1/8000 + \varepsilon})
  \end{align*}
  for any $\varepsilon > 0$, and 
  where $F_{\Phi, V, \theta}$ is defined in Proposition \ref{prop3}, and
  $$
  H_{\Phi, V, k, \theta, \Delta} (\boldsymbol{t}, \boldsymbol{u}) 
  := \int_{\boldsymbol{\xi} \in \mathbb{R}^k} \prod_{i = 1}^{k} 
  F_{\Phi, V, \theta}(\xi_i, u_i) 
  e ( - \langle \boldsymbol{t}, \boldsymbol{\xi} \rangle ) 
  w(\Delta^3 \boldsymbol{\xi}) d \boldsymbol{\xi}
  $$
  where $w(\boldsymbol{\xi}) := w(\xi_1 + \ldots + \xi_k)$.
\end{prop}

%Since the function $F_{\Phi, W}(\xi, \eta)$ starts to decay as soon as $|\xi| \gg 1$ or $|\eta| \gg 1$ we expect (on probalistic grounds) that the right-hand side in \eqref{eq:form} contains on average $O(1)$ terms. It turns out that this is true in $L^{p}$ for $p \leq 2$ but not $L^p$ with $p > 2$.
\section{Preliminary Lemmas}
First, we will quantify how well the Jutila measure 
\eqref{def Jutila measure} approximates
Lebesgue measure. To that end,
an $L^2$-argument suffices
which naturally requires us to study an associated
pair correlation function. Let us detail this now.
For a finite set $\mathcal{S}\subseteq [0,1)$ of $L$ elements,
we define the (smoothed) pair correlation function 
$$
R_2(\mathcal{S},\Psi,\lambda):= 
\frac{1}{L} \sum_{\substack{s_1, s_2\in \mathcal{S}\\ s_1\neq s_2}} 
\Psi\Big(\frac{\Vert s_1 - s_2 \Vert_{\mathbb{R}/\mathbb{Z}}}
{\lambda L^{-1}}\Big)
$$
where $\Vert \cdot \Vert_{\mathbb{R}/\mathbb{Z}}$
denotes the distance to the nearest integer.
Our first lemma studies 
the pair correlation function of rough and square-free
Farey fractions;
compare \cite[Thm. 2]{boca2005correlations}.
\begin{lemma}\label{le:paircorrel}
 Let $\Delta, N \geq 10$ and set 
 $Q := \Delta \sqrt{N}$. 
  Let $\mathcal{Q}=\mathcal{Q}_{\Delta,N}$ 
  be as in \eqref{eq:qu}. Suppose $\Psi$
  is a non-negative, even, 
  smooth function that is compactly supported in 
  $[-1 / 10,1 / 10]$.
  Define $\mathcal{S}=\{a/q: 
  1\leq a < q, (a,q)=1, q\in \mathcal{Q}\}$. Then,
  \begin{equation}\label{eq:paircorrelasymp}
    R_2(\mathcal{S},\Psi,\lambda) = 
    \lambda \widehat{\Psi}(0) + O(\Delta^{-2000})  
  \end{equation}
uniformly for all $\Delta\ll \lambda \leq \Delta^{100}$ as $N \rightarrow \infty$.
\end{lemma}
\begin{proof}
To analyze $R_2(\mathcal{S},\Psi,\lambda)$,
we decompose the off-diagonal in $\mathcal{S}^2$ 
into the sets $\mathcal{S}^2(g)$ consisting of all
$(s_1,s_2)=(a_1/q_1,a_2/q_2)\in  
\mathcal{S}^2$ with $g=(q_1,q_2)$ and $s_1\neq s_2$.
Then, 
\begin{equation}\label{eq gcd decomp}
R_2(\mathcal{S},\Psi,\lambda)
= 
\frac{1}{L} 
\sum_{g\geq 1} \sum_{\substack{(\frac{a_1}{q_1},
\frac{a_2}{q_2})\in \mathcal{S}^2(g)}} 
\Psi\Big(\frac{\Vert \frac{a_1}{q_1} - 
\frac{a_2}{q_2} \Vert_{\mathbb{R}/\mathbb{Z}}}
{\lambda L^{-1}}\Big).
\end{equation}
We claim $\mathcal{S}^2(g)$ contributes zero to \eqref{eq gcd decomp} if $g>1$. To verify this,
we will show that there
\begin{equation}\label{eq lower bound on fractions}
\frac{\Vert \frac{a_1}{q_1} - 
\frac{a_2}{q_2} \Vert_{\mathbb{R}/\mathbb{Z}}}
{\lambda L^{-1}} > \Delta^{1000}
\end{equation}
is inevitably outside of 
$\mathrm{supp}(\Psi)\subseteq [-1/10,1/10]$.
Let $(a_1/q_1,a_2/q_2)\in \mathcal{S}^2(g)$. We start 
by writing $q_1=q_1' g$ and $q_2=gq_2'$ so that $(q_1',q_2')=1$.
Notice
$$
\frac{a_1}{q_1} - 
\frac{a_2}{q_2}
= 
\frac{a_1q_2' - a_2 q_1'}{g q_1'q_2'}.
$$
The left-hand side is non-vanishing (as we
are on the off-diagonal $a_1/q_1\neq a_2/q_2$).
Hence, the integer $a_1q_2' - a_2 q_1'$ is non-zero
ensuring 
$\vert a_1q_2' - a_2 q_1'\vert \geq 1$ and
$$
\Big \vert 
\frac{a_1}{q_1} - 
\frac{a_2}{q_2}
\Big \vert \geq \frac{1}{g q_1'q_2'}.
$$
By the roughness stipulation in \eqref{eq:qu}, 
if $g>1$ then $g> \Delta^{2001}$.
Note that $L\leq \Delta^5 N$ and that $q_{1}', q_{2}' \leq 2 Q / g$.
Thus, 
$$
\frac{1}{g q_1'q_2'} \geq \frac{g}{4 Q^{2}} \geq \frac{\Delta^{2001}}{4 \Delta^{2} N}
> \frac{\Delta^{1800}}{L}.
$$
As $\lambda \leq \Delta^{100}$,
we deduce \eqref{eq lower bound on fractions}.
Consequently, \eqref{eq gcd decomp} simplifies to
$$
R_2(\mathcal{S},\Psi,\lambda)
= 
\frac{1}{L}  \sum_{\substack{(\frac{a_1}{q_1},
\frac{a_2}{q_2})\in \mathcal{S}^2(1)}} 
\Psi\Bigg(\frac{\big \Vert \frac{a_1}{q_1} - 
\frac{a_2}{q_2} \big\Vert_{\mathbb{R}/\mathbb{Z}}}
{\lambda L^{-1}}\Bigg).
$$
Fix co-prime
$q_1,q_2\in \mathcal{Q}$.
Let $\mathfrak{q}:=q_1q_2$.
Given an integer $m\geq 1$, we use the abbreviation $\mathbb{Z}_{m}^*:=\{1\leq h\leq m: (h,m)=1\}$.
The map
$(a_1,a_2)\mapsto a_1 q_2 - q_1 a_2 \,\mathrm{mod}\, \mathfrak{q}$
is a bijection 
between $\mathbb{Z}_{q_1}^*\times \mathbb{Z}_{q_2}^*$
and $\mathbb{Z}_{\mathfrak{q}}^*$. Moreover if 
$ a_1 q_2 - q_1 a_2 \,\mathrm{mod}\, \mathfrak{q} = h$, then
$$
\frac{1}{\lambda L^{-1}}
\Big \Vert \frac{a_1}{q_1} - 
\frac{a_2}{q_2} \Big\Vert_{\mathbb{R}/\mathbb{Z}}
= 
\min\Big( \frac{L}{\lambda}\cdot \frac{h}{\mathfrak{q}},
 \frac{L}{\lambda}\cdot  \frac{\mathfrak{q}-h}{\mathfrak{q}} \Big).
$$
Let $\mathcal{I}:=\mathrm{supp}(\Psi) \cap \mathbb{R}_{\geq 0}$. Next, we 
%show
%$\Psi\big(\frac{\Vert h/\mathfrak{q} 
%\Vert_{\mathbb{R}/\mathbb{Z}}}{\lambda L^{-1}}\big)$
%is non-zero
%if and only if $h$ is in one of two intervals. 
%Let
%$B_-:= \inf(\mathrm{supp}(\Psi)\cap\mathbb{R}_{\geq 0})$, 
%$B_+:= \sup(\mathrm{supp}(\Psi)\cap\mathbb{R}_{\geq 0})$,
%and $\{1\leq h\leq \mathfrak{q}\}=:\mathbb{Z}_\mathfrak{q}$.
observe that
$$
\Psi \Big ( \frac{L}{\lambda} \cdot \Big \| \frac{h}{\mathfrak{q}} \Big \|_{\mathbb{R} / \mathbb{Z}} \Big ) \neq 0
$$
if and only if,
$$
\min \Big (h, \mathfrak{q} - h \Big ) \in \frac{\mathfrak{q} \lambda}{L} \cdot \mathcal{I}.
$$
Thus,
$$
\Bigg\{h \in \mathbb{Z}_\mathfrak{q}: 
\Psi\Big( \frac{L}{\lambda}
\Big \Vert \frac{h}{\mathfrak{q}} 
\Big\Vert_{\mathbb{R}/\mathbb{Z}}\Big)
\neq 0
\Bigg\}
=
\Big(\frac{\mathfrak{q} \lambda}{L}
\mathcal{I}\cup 
\big(\mathfrak{q} - 
\frac{\mathfrak{q} \lambda}{L} \mathcal{I},
\mathfrak{q}\big)\Big)
\cap \mathbb{Z}_\mathfrak{q}.
$$
Each of the two intervals on the right-hand side has length comparable to
$\lambda/5 \leq \Delta^{100}$.
Using their location, we can add co-primality conditions to
both sides of the equation:
$$
\Bigg\{h \in \mathbb{Z}_\mathfrak{q}^*: 
\Psi\Big( \frac{L}{\lambda} \Big \Vert  
\frac{h}{\mathfrak{q}} \Big\Vert_{\mathbb{R}/\mathbb{Z}}\Big)
\neq 0
\Bigg\}
=
\Big(\frac{\mathfrak{q}\lambda}{L}
\mathcal{I}\cup 
\big(\mathfrak{q} - 
\frac{\mathfrak{q} \lambda}{L}\mathcal{I},
\mathfrak{q}\big)\Big)
\cap \mathbb{Z}^*_\mathfrak{q}.
$$
To proceed, it is helpful to have more information 
on $L$ at hand. To this end, we write 
$$
L = \sum_{q \in \mathcal{Q}} \varphi(q) = \sum_{q\in \mathcal{Q}} \sum_{d\mid q} \frac{q}{d} \mu (d)
=\sum_{1 \leq  d \leq 2Q} 
 \frac{\mu (d)}{d}
\sum_{
\substack{q\in \mathcal{Q}\\
d\mid q}}  q
= 
\sum_{q\in \mathcal{Q}}  q
+ 
E
$$
where $Q = \Delta \sqrt{N}$ and the contribution $E$ (from the range $d>1$) is readily estimated via
$$
E= 
\sum_{\Delta^{2001} \leq  d \leq 2Q} 
 \frac{\mu (d)}{d}
\sum_{
\substack{q\in \mathcal{Q}\\
d\mid q}}  q\ll 
\sum_{\Delta^{2001} \leq d \leq 2Q} 
 \frac{1}{d}
\Big( \frac{Q}{d} + 1 \Big) Q
\ll \frac{Q^2}{\Delta^{2001}}.
$$
As
$\#\mathcal{Q}\asymp \frac{Q}{\log \Delta}$,
we infer 
$\sum_{q\in \mathcal{Q}}  q \asymp \frac{Q^2}{\log \Delta}
\asymp \frac{\Delta^2 N}{\log \Delta} $.
%\begin{equation}\label{eq sum q}
%\sum_{q\in \mathcal{Q}}  q \asymp \frac{Q^2}{\log \Delta}.
%\end{equation}
Consequently,
\begin{equation}\label{eq: approx to L}
L= (1+ O(\Delta^{-2000})) \sum_{q\in \mathcal{Q}}  q.
\end{equation}
By $\lambda \gg \Delta$, we infer from
\eqref{eq: approx to L} that
$\lambda \mathfrak{q}/L\gg \Delta$.
Poisson summation implies
\begin{align*}
\sum_{\substack{h\in \frac{\mathfrak{q} \lambda}{L}\mathcal{I}\\
(h,\mathfrak{q})=1}}
\Psi\Big( \frac{L}{\lambda}\Big \Vert  
\frac{h}{\mathfrak{q}} \Big\Vert_{\mathbb{R}/\mathbb{Z}}\Big)
= 
\frac{1}{2}
\sum_{h\in \mathbb{Z}}
\Psi\Big(\frac{Lh}{\lambda\mathfrak{q}} \Big)
= \frac{1}{2} 
\frac{\lambda\mathfrak{q}}{L}
\widehat{\Psi}(0)
+ O_A(\Delta^{-A})
\end{align*}
for any $A>0$.
In the same way,
\begin{align*}
\sum_{\substack{h\in (\mathfrak{q}- 
\frac{\mathfrak{q} \lambda}{L}\mathcal{I}, \mathfrak{q})\\
(h,\mathfrak{q})=1}}
\Psi\Big(\frac{L}{\lambda} \Big \Vert  
\frac{h}{\mathfrak{q}} \Big\Vert_{\mathbb{R}/\mathbb{Z}}\Big)
& = \frac{1}{2} 
\frac{\lambda\mathfrak{q}}{L}
\widehat{\Psi}(0)
+ O_A(\Delta^{-A}).
\end{align*}
The upshot is that we can approximately compute the summations over the 
$a_i$ as follows
$$
\sum_{\substack{(a_i,q_i)=1\\
i=1,2}} 
\Psi\Bigg(\frac{\big \Vert \frac{a_1}{q_1} - 
\frac{a_2}{q_2} \big\Vert_{\mathbb{R}/\mathbb{Z}}}
{\lambda L^{-1}}\Bigg) = 
\frac{\lambda \mathfrak{q}}{L} \widehat{\Psi}(0) + O_A(\Delta^{-A}).
$$
As a result, 
\begin{equation}\label{eq circle method interm}
R_2(\mathcal{S},\Psi,\lambda)
= 
\lambda \widehat{\Psi}(0) 
\sum_{\substack{(q_1,q_2)\in \mathcal{Q}^2\\
(q_1,q_2)=1}}
\frac{q_1q_2}{L^2} + O_A(\Delta^{-A}).
\end{equation}
We write
\begin{align}\label{eq interm q sum}
\sum_{\substack{(q_1,q_2)\in \mathcal{Q}^2\\
(q_1,q_2)=1}} \frac{q_1q_2}{L^2}
= 
\Bigg(
\sum_{\substack{q\in \mathcal{Q}}} \frac{q}{L}
\Bigg)^2
+ O\Bigg( 
\sum_{\Delta^{2001}< g < 2Q}
\sum_{\substack{(q_1,q_2)\in \mathcal{Q}^2\\
(q_1,q_2)=g}} \frac{q_1q_2}{L^2}
\Bigg).
\end{align}
For each $g<2Q$, there are 
$O(\frac{Q^2}{g^2})$
many $(q_1,q_2)\in \mathcal{Q}^2$ 
with $(q_1,q_2)=g$. Thus, 
$$
\sum_{\Delta^{2001}< g < 2Q}
\sum_{\substack{(q_1,q_2)\in \mathcal{Q}^2\\
(q_1,q_2)=g}} \frac{q_1q_2}{L^2}
\ll 
\frac{1}{\Delta^{2000}}.
$$
By \eqref{eq: approx to L} we know
$\sum_{\substack{q\in \mathcal{Q}}} \frac{q}{L}
= 1 + O(\Delta^{-2000})$.
Using this information in \eqref{eq interm q sum} 
yields
$$
\sum_{\substack{(q_1,q_2)\in \mathcal{Q}^2\\
(q_1,q_2)=1}} \frac{q_1q_2}{L^2}
= 
1
+ O(\Delta^{-2000}).
$$
In view of \eqref{eq circle method interm}, we conclude 
$R_2(\mathcal{S},\Psi,\lambda)= 
\lambda \widehat{\Psi}(0) + O(\Delta^{-2000})$.
\end{proof}
Next, we need a refined version of Jutila's circle method,
see \cite{Mayank}. 
\begin{lemma} \label{le:mayank}
  Let $\Delta, N \geq 10$, 
  %$Q := \Delta \sqrt{N}$, 
  and $\mathcal{Q}=\mathcal{Q}_{\Delta,N}$ be as in \eqref{eq:qu}.
  %Let $\mathcal{Q}$ be a set of integers contained in $[Q, 2Q]$. 
  Suppose $\phi$
  is a non-negative, even, smooth function that is
  compactly supported in $[-1 / 100,1 / 100]$ with 
  $\int_{\mathbb{R}} \phi (x) \,dx = 1$. Set
  $$
  \widetilde{\chi}(\alpha) := \frac{N}{L}
  \sum_{\substack{(a,q) = 1 \\ q \in \mathcal{Q}}} 
  \phi \Big ( N 
  \Big ( \alpha - \frac{a}{q} \Big ) \Big ) \ , 
  \ L := \sum_{q \in \mathcal{Q}} \varphi(q). 
  $$
  Then,
  $$
  \int_{0}^{1} |1 - \widetilde{\chi}(\alpha)|^2 
  d \alpha 
  \ll \frac{1}{\Delta}.
  $$
\end{lemma}
\begin{proof}
Put $\mathcal{S}:=\{a/q: 
  1\leq a < q, (a,q)=1, q\in \mathcal{Q}\}$.
By Poisson summation, 
$$\widetilde{\chi}(\alpha) = 
\frac{1}{L}\sum_{\ell \in \mathbb{Z}} 
\widehat{\phi}\Big ( \frac{\ell}{N} \Big ) 
c(\ell) e( - \ell \alpha) 
\quad \mathrm{where}
\quad 
c(\ell) := \sum_{s\in \mathcal{S}} 
e ( \ell s ).
$$
Since $\phi$ is even, $\widehat{\phi}$ is even and real-valued.
The zero mode, $\ell=0$, contributes one to the right-hand side.
Let $X:=\int_{0}^{1} | 1 - \widetilde{\chi}(\alpha) |^2 
d \alpha $. Plancherel's theorem implies 
\begin{align*}
X & =  \frac{1}{L^2} 
\sum_{\ell \neq 0}  
\Big(\widehat{\phi} \Big(\frac{\ell}{N}\Big) \Big)^2
%\bigg \vert \sum_{\substack{(a,q)=1\\ q \in \mathcal{Q}}} 
%e \Big ( \frac{a \ell}{q} \Big ) 
%\bigg \vert^2
\vert c(\ell)\vert^2.
\end{align*}
To proceed, we consider the even function $\Psi:=\phi*\phi$.
Its relevance is that $\widehat{\Psi}(\xi)=(\widehat{\phi}(\xi))^2$.
Expanding the square and 
interchanging the order of summation yields 
$$
X = 
\frac{1}{L^2}  \sum_{s_1,s_2 \in \mathcal{S}} 
\sum_{\ell \neq 0}
\widehat{\Psi}\Big(\frac{\ell}{N}\Big)
e ( \ell (s_1 - s_2)). 
$$
Let $\lambda :=L/N$, and notice $L= \#\mathcal{S}$.
The contribution of the diagonal $s_1=s_2$  to $X$ is
$$
\frac{\#\mathcal{S}}{L^2}  
\sum_{\ell \neq 0}
\widehat{\Psi}\Big(\frac{\ell}{N}\Big)
\ll \frac{1}{\lambda}.
$$
From the off-diagonal, 
we add and subtract $\widehat{\Psi}(0)=1$.
Poisson summation over $\ell$ implies
\begin{equation}\label{eq X in terms of pair correl}
X = \frac{R_2(\mathcal{S},\Psi,\lambda)}{\lambda} - 1 + 
O_A\Big(\frac{1}{\lambda} + \Delta^{-A}\Big).
\end{equation}
From %\eqref{eq sum q}
%and 
\eqref{eq: approx to L}
we see $\lambda \gg \Delta^2/\log \Delta$.
Applying Lemma \ref{le:paircorrel}
completes the proof.
\end{proof}

We will frequently appeal to the following form of the Poisson summation formula. 

  \begin{lemma}[Poisson summation] \label{le:poisson} 
  Let $q \geq 1$.
    Let $f$ be a Schwartz function and $K$ be any 
  $m$-periodic function.
    Then,
    $$ \sum_{n \in \mathbb{Z}} K(n) f(n) = 
    \sum_{\ell \in \mathbb{Z}} 
    \widehat{f} \Big ( \frac{\ell}{m} \Big ) 
    (\mathcal{F}_m^{-1}K)(\ell) $$
    where
    $$ (\mathcal{F}_m^{-1} K)(\ell) 
    := \frac{1}{m}\sum_{x \mmod{m}} 
    K(x) e \Big ( \frac{\ell x}{m} \Big ) $$
    is the inverse of the discrete Fourier transform modulo $m$. 
  \end{lemma}

  In order to use Poisson summation effectively
   we also need two auxiliary results.

\begin{lemma}\label{le:dist}
  Let $v \neq 0$ and let $f : \mathbb{R} \rightarrow \mathbb{R}$ be a smooth compactly support Schwarz function. Then, 
  $$
  \int_{\mathbb{R}} e(-v y^2) f(y) dy = \frac{e(-\tfrac {\mathrm{sgn}(v)}{8})}{\sqrt{2 |v|}} \int_{\mathbb{R}} e \Big (  \frac{\xi^2}{4 v} \Big ) \widehat{f}(\xi) d \xi.
  $$
\end{lemma}
\begin{proof}
  First assume $v > 0$. Pick an $\varepsilon > 0$. 
  We apply Plancherel to
  $
  \int_{\mathbb{R}} e(-v y^2) e^{-\varepsilon y^2} e^{\varepsilon y^2} f(y) dy.
  $
  By completing the square we find
 \begin{align*}
   \int_{\mathbb{R}} e(- v y^2) e^{-\varepsilon y^2} e(- \xi y) dy & = e^{\frac{(2\pi i\xi)^2}{4 (\varepsilon + 2 \pi i v)}} \int_{\mathbb{R}} e^{- (\varepsilon + 2\pi i v) \cdot ( y + \frac{2\pi i\xi}{2\varepsilon + 4\pi i v}  )^2} dy \\ & =
  e \Big ( \frac{2 \pi i \xi^2}{4 (\varepsilon + 2 \pi i v)} \Big ) \frac{\sqrt{\pi}}{\sqrt{\varepsilon + 2 \pi i v}}.
  \end{align*}
  Letting $\varepsilon \downarrow 0$ implies the claim. 
  To prove the result for $v < 0$ it suffices to take the conjugate on both sides of the equation that 
  we just established in the case $v > 0$. 
\end{proof}

Next, we record a classical result on Gauss sums
which is the discrete analogue of Lemma \ref{le:dist}. 
\begin{lemma}[Evaluating quadratic Gauss sums]\label{le:quadratic Gauss sums}
  For integer $v > 0$ and $u \in \mathbb{Z}$ we have
  \[ 
\mathcal{G}(u; 4v) := \sum_{0\leq x <4v} 
e \Big ( \frac{x^2+ux}{4v}\Big) = 
 \sqrt{2} \sqrt{4 v} e \Big ( \frac{1}{8} \Big ) 
 e \Big ( - \frac{(u/2)^2}{4 v} \Big ) \mathbf{1}_{2\mathbb{Z}}(u).
\]
\end{lemma}
\begin{proof}
First, we suppose that $u$ is even. 
Completing the square yields
$$\mathcal{G}(u; 4v) = e (-(u/2)^2/(4 v)) \mathcal{G}(0; 4v).$$
Thanks to \cite[Theorem 3.4]{IwaniecKowalski} we know
$\mathcal{G}(0; 4v)=(1 + i) \sqrt{4v}$.
Notice $(1 + i) = \sqrt{2} e(\tfrac 18)$.
Now let $u$ be odd. Choose $\nu \geq 2$ 
so that $4v/2^{\nu}$ is an odd integer.
Then $x^2+ux$ runs through all even residue classes modulo $2^\nu$ exactly twice. By geometric summation,
$\mathcal{G}(u; 2^\nu) =0$.
By multiplicativity 
$\mathcal{G}(u; 4v) =0$.
\end{proof}

We need the following auxiliary results on generalized Gauss sums
\[
\mathcal{G}(a,b;c):=\sum_{x\,\mathrm{mod}\,c}e \Big (\frac{ax^{2}+bx}{c} \Big ).
\]
If the leading coefficient $a=1$, then we simply write $\mathcal{G}(b;c):=\mathcal{G}(1,b;c)$.
Later-on we need:
\begin{lemma}\label{le:Gauss sums vanish}
If $d:=\gcd(a,c)>1$, then $\mathcal{G}(a,b;c)=0$ as long as $d\nmid b$.
\end{lemma}
\begin{proof}
Write $a=da'$ , and $c=dc'$. We parameterize $0\leq x<c$ as $x=jc'+r$
with $0\leq j<\frac{c}{c'}=d$ and $0\leq r<c'$. Thus 
\begin{align*}
\mathcal{G}(a,b;c) & =\sum_{0\leq j<d}\sum_{0\leq r<c'}e \Big (\frac{a'(jc'+r)^{2}}{c'}+\frac{b (j c' + r)}{c} \Big )\\
 & =\sum_{0\leq j<d}\sum_{0\leq r<c'}e \Big (\frac{a'r^{2}}{c'}+\frac{b j}{d}+\frac{b r}{c}\Big ).
\end{align*}
Interchanging the order of summation and using that $d\nmid b$, we
see that the right hand side vanishes.
\end{proof}

Throughout we will need a good understanding of the decay of the function $F_{\Phi, V, \theta}$. The subsequent lemma will be often used. 

\begin{lemma}\label{le: decay bounds} 
  Let $0 \leq V \leq 1$ be a smooth function compactly supported in $(1/2, 3)$ and $0 \leq \Phi \leq 1$ a smooth function with compact support.
  Let
  $$
  F_{\Phi, V, \theta}(\xi, \eta) 
  := \int_{0}^{\infty} 
  \widehat{\Phi}(2 \xi x) 
  e(2 \xi \theta x) x V(x^2) 
  e \Big ( - \frac{\eta x}{2} \Big ) dx.  
  $$
  We have, for any $A > 10$ and $\ell \geq 0$,  
  $$
  \frac{\partial^{\ell}}{\partial^{\ell} \xi} F_{\Phi, V, \theta}(\xi, \eta) \ll_{A, \ell, \Phi} \frac{1 + |\theta|^{\ell}}{1 + |\xi|^{A}}
  $$
  and for $|\eta| > 50 |\xi \theta|$, 
  $$
  F_{\Phi, V, \theta}(\xi, \eta) \ll_{A, \Phi, V} \frac{1 + |\xi|^A}{1 + |\eta|^{A}}.
  $$
  Finally, we have the trivial bound, 
  $$
  \frac{\partial^{\ell}}{\partial^{\ell} \xi} \frac{\partial^v}{\partial^v \eta} F_{\Phi, V, \theta}(\xi, \eta) \ll_{\Phi, V, \ell, v} 1 + |\theta|^{\ell}.
  $$
\end{lemma}
\begin{proof}
The first bound follows by noticing that $V(x)$ localizes $x \asymp 1$, and thus, inside the integral, 
$$
|\widehat{\Phi}(2 \xi x)| \ll_{A, \Phi} \frac{1}{1 + |\xi|^{A}}
$$
for any given $A > 10$. The second bound follows from integrating by parts the phase $e(2 \xi \theta x - \eta x / 2)$. This yields, 
$$
|F_{\Phi, V, \theta}(\xi, \eta)| \ll_{A, \Phi, V} \frac{1 + |\xi|^{A}}{1 + |2 \xi \theta - \eta|^A}
$$
the result then follows from the assumption that $|\eta| > 50 |\xi \theta|$. 
\end{proof}

Finally, we will need to understand averages over $q \in \mathcal{Q}_{\Delta, N}$ weighted by $\varphi(q)$. This is the content of the next Lemma. 
\begin{lemma} \label{le:comput} 
  Let $\Delta, N > 10$ with $\Delta^{2001} \leq \log N$. We
  define, as in \eqref{eq:qu}, 
  $$
  \mathcal{Q}_{\Delta, N} := \Big \{ a b \in [Q, 2Q]: Q^{1/1000} \leq a \leq 2 Q^{1/1000}, p | a b \implies p > \Delta^{2001}, \mu^2(a b) = 1 \Big \}
  $$
  with $Q := \Delta \sqrt{N}$. Then, for any $\varepsilon \in (0, 1)$, 
$$
\sum_{\substack{q \in \mathcal{Q}_{\Delta, N} \\ Q \leq q \leq (1 + \varepsilon) Q}} \frac{\varphi(q)}{q} = \varepsilon \sum_{q \in \mathcal{Q}_{\Delta, N}} \frac{\varphi(q)}{q} + O(\sqrt{Q}). 
$$
and 
$$
\sum_{q \in \mathcal{Q}_{\Delta, N}} \varphi(q) = \frac{3}{2} \cdot \Delta \sqrt{N} \sum_{q \in \mathcal{Q}_{\Delta, N}} \frac{\varphi(q)}{q}  + O(\Delta Q^{3/2}).
%\frac{\sqrt{N} \sum_{q \in \mathcal{Q}_{\Delta, N}} \frac{\varphi(q)}{q}}{\sum_{q \in \mathcal{Q}_{\Delta, N}} \varphi(q)} = 
$$
\end{lemma} 
\begin{proof}
  By definition of $\mathcal{Q}_{\Delta, N}$,  
  \begin{align} \label{eq:secondmove}
  \sum_{\substack{q \in \mathcal{Q}_{\Delta, N} \\ Q \leq q \leq (1 + \varepsilon) Q}} \frac{\varphi(q)}{q} & = \sum_{\substack{Q^{1/1000} \leq a \leq 2 Q^{1/1000} \\ p | a \implies p > \Delta^{2001}}} \frac{\mu^2(a)\varphi(a)}{a} \sum_{\substack{(a,b) = 1 \\ Q \leq a b \leq (1 + \varepsilon) Q \\ p | b \implies p > \Delta^{2001}}}\frac{\mu^2(b)\varphi(b)}{b}.
  \end{align}
  By M\"{o}bius inversion, 
  $$
  \mathbf{1} \Big ( p | b \implies p \nmid a , p > \Delta^{2001} \Big ) \frac{\mu^2(b) \varphi(b)}{b} = \sum_{d | b} f^{\star}(d)
  $$
  with $f^{\star}$ a multiplicative function such that $f^{\star}(p) = -1$, $f^{\star}(p^k) = 0$, $k \geq 2$ for $p | a$ or $p \leq \Delta^{2001}$ and $f^{\star}(p) = - \frac{1}{p}$, $f^{\star}(p^2) = -1 + \frac{1}{p}$, $f^{\star}(p^k) = 0$, $k \geq 3$ otherwise. Therefore, 
  \begin{align} \label{eq:firstmove}
  \sum_{\substack{(a,b) = 1 \\ Q \leq a b \leq (1 + \varepsilon) Q \\ p | b \implies p > \Delta^{2001}}} \frac{\mu^2(b)\varphi(b)}{b} & = \sum_{d} f^{\star}(d) \Big ( \sum_{Q \leq a d b \leq (1 + \varepsilon)Q } 1 \Big ) \\ \nonumber & = \sum_{d} f^{\star}(d) \Big ( \varepsilon \sum_{Q \leq a d b \leq 2Q} 1 \Big ) + O(\sqrt{Q}) \\ \nonumber & = \varepsilon \sum_{\substack{(a,b) = 1 \\ Q \leq a b \leq 2 Q \\ p | b \implies p > \Delta^{2001}}} \frac{\mu^2(b)\varphi(b)}{b} + O(\sqrt{Q}).
  \end{align}
  Combining \eqref{eq:secondmove} and \eqref{eq:firstmove} we obtain the first claim. The second claim follows by integration by parts, since, 
  \begin{align*}
  \sum_{q \in \mathcal{Q}_{\Delta, N}} \varphi(q) & =  \int_{1}^{2} (\Delta \sqrt{N} u) d \Big ( \sum_{\substack{q \in \mathcal{Q}_{\Delta, N} \\ q \in [\Delta \sqrt{N}, u \Delta \sqrt{N}]}} \frac{\varphi(q)}{q} \Big ) \\ & = \Delta \sqrt{N} \sum_{q \in \mathcal{Q}_{\Delta, N}} \frac{\varphi(q)}{q} \int_{1}^{2} u du + O(\Delta Q^{3/2})
  \end{align*}
  as needed. 
%  On combining the above two equations, 
% $$
%  \sum_{\substack{q \in \mathcal{Q}_{\Delta, N} \\ M \leq q \leq (1 + \varepsilon) M}} \frac{\varphi(q)}{q} = \varepsilon \sum_{\substack{q \in \mathcal{Q}_{\Delta, N}}} \frac{\varphi(q)}{q} + O(Q^{\delta} \log Q)
%  $$
%  It follows from this that, 
%  $$
 % \frac{1}{\widetilde{L}_{\delta, \nu} \sqrt{N}} \sum_{q \in \mathcal{Q}_{\Delta, N}} \varphi(q) \mathfrak{F}_{\boldsymbol{t}, \boldsymbol{u}} \Big ( \frac{q}{\sqrt{N}} \Big ) = \frac{1}{\Delta} \int_{\Delta}^{2 \Delta} u \mathfrak{F}_{\boldsymbol{t}, \boldsymbol{u}} (u) du + O(Q^{-1 + 1/1000} \log Q)
%  $$
  
\end{proof}
\section{Proof of Proposition \ref{prop1}}

For a smooth function $\phi$ compactly supported in $[-1 / 100, 1 / 100]$ let 
$$
\widetilde{\chi}(\alpha) := 
\frac{N}{L}
  \sum_{\substack{(a,q) = 1 \\ q \in \mathcal{Q}_{\Delta, N}}} 
  \phi \Big ( N 
  \Big ( \alpha - \frac{a}{q} \Big ) \Big ) \ , 
  \ L := \sum_{q \in \mathcal{Q}_{\Delta, N}} \varphi(q). 
  $$
By Lemma \ref{le:mayank} 
\begin{align*}
\Big | \int_{0}^{1} (1 - \widetilde{\chi}(\alpha) ) 
\mathbf{1} ( \mathcal{N}(\alpha; s, N) = 0 ) 
d \alpha \Big | \leq \Big ( \int_{0}^{1} |1 - \widetilde{\chi}(\alpha) |^2 d \alpha \Big )^{1/2}
\ll \frac{1}{\Delta^{1/2}}.
\end{align*}
concluding the proof. %Notice that the factor $\log \Delta$ comes from the fact that $\mathcal{Q}_{\Delta, N}$ consists
%of integers $q$ such that all prime factors of $q$ are larger than $\Delta^{2001}$. 

\section{Proof of Proposition \ref{prop3}}

%  \begin{lemma}\label{le:comput}
%    Let $0 \leq V \leq 1$ be a smooth function 
%    compactly supported in $(0, 4)$. 
%    Let $\Phi$ be smooth and compactly supported. 
%    Let $a$ and $q$ be coprime.
    % and $q \asymp \Delta \sqrt{N}$.
%    Then, 
%    \begin{align*}
%    R_{\Phi, V} \Big ( \frac{a}{q} - \theta; N \Big ) 
%    = \widehat{\Phi}(0) \widehat{V}(0) + 
%    \sum_{\substack{v,u \neq 0 \\ 
%    u \equiv 2 v^{\star} a \mmod{q^{\star}}}} 
%    e \Big ( - \frac{\overline{q^{\star}}^2 u^2}{4 v} 
%    + \frac{(u+av^{\star}) a \overline{q^{\star}}^2}{d} \Big ) 
%    & F_{\Phi, V,\theta} 
%    \Big ( \frac{v}{\sqrt{N}}, \frac{u}{q^{\star} / \sqrt{N}} \Big ) \\ & + O(N^{-1/2} \| V \|_{C^2})
%    \end{align*}
%    where $d := (v,q)$,
%    $v^{\star} := v / d$,
%    $q^{\star} := q / d$ and 
%    $F_{\Phi,V,\theta}$ is as in Proposition \ref{prop3}.
%    %$$
%    %F_{\Phi, V}(\xi, \eta) := \int_{0}^{\infty} \widehat{\Phi}(2 \xi x) x V(x^2) e \Big ( - \frac{\eta x}{2} \Big ) d x.
%    %$$
%  \end{lemma}
%  \begin{proof}
    Without loss of generality we can assume that $N$ is sufficiently large with respect to the support of $\Phi$ so that $\Phi(x) = 0$ for all $x \geq N$. Recall also that $\Phi(x)  = 0$ for $x < 0$. We start by opening $\Phi(N \| \cdot \|)$ into a Fourier series (recall that $\| x \| := x \mmod{1}$) thus
    $$
    R_{\Phi, V} \Big (\frac{a}{q} - \theta; N \Big ) 
    = \frac{1}{N} \sum_{\ell \in \mathbb{Z}} 
    \widehat{\Phi} \Big ( \frac{\ell}{N} \Big ) 
    e \Big ( - \frac{\ell a}{q} \Big ) 
    e(\ell \theta)
    \sum_{n \in \mathbb{Z}} e(\ell \sqrt{n}) 
    V \Big ( \frac{n}{N} \Big ).
    $$
    We will now transform this sum by a double application of
    Poisson summation.
    \subsubsection{The sum over $n$}    
    Throughout, let $A>1$ be an arbitrarily large constant. 
    We seek to evaluate
    $$
    S(\ell) = \sum_{n \in \mathbb{Z}} e(\ell \sqrt{n}) V \Big ( \frac{n}{N} \Big ),
    $$
    uniformly in $\ell$. 
    For $\ell = 0$ we find by Poisson summation
    $
    S(0) = N \widehat{V}(0) + 
    O_{V, A}(N^{-A}).
    $
    Consider now $\ell \neq 0$. Lemma \ref{le:poisson} yields
    $$
    S(\ell) = \sum_{v \in \mathbb{Z}} I(\ell, v)
    \,\, \mathrm{where} \,\,
    I(\ell,v):=\int_{\mathbb{R}} e(\ell \sqrt{y} - v y) V \Big ( \frac{y}{N} \Big ) dy. 
    $$
    In the case $v = 0$ we see $I(\ell, 0) \ll_{A} N^{-A}$. 
%    Let $W(y) := V(y^2) y$. By a change of variables (i.e. $y \mapsto y^2$), we conclude 
%    $I(\ell,0) = 2 N \widehat{W}(\sqrt{N} \ell)\ll_{A} N^{-A}$ for any given $A > 0$. 
    Consider now $v\neq 0$. Integration by parts shows 
    $I(\ell,v) \ll_{A} N^{-A}$ if 
    $\mathrm{sgn}(v) \neq \mathrm{sgn}(\ell)$.
    Suppose $\mathrm{sgn}(v) = \mathrm{sgn}(\ell)$.  
    Let $$
    W(y) := V(y^2) y.
    $$
    By a change of variables (i.e. $y \mapsto y^2$) and completing the square,
    \begin{align} \label{eq:equation}
      I(\ell,v)
      = 2 \sqrt{N} e \Big ( \frac{\ell^2}{4 v} \Big ) 
      \int_{\mathbb{R}} e(- v y^2) W \Big ( \frac{y + \frac{\ell}{2 v}}{\sqrt{N}} \Big ) d y.
    \end{align}
    Applying Lemma \ref{le:dist}, 
    combined with a change of variables
    (i.e. $\xi \mapsto \xi/\sqrt{N}$), implies
    \begin{equation} \label{eq:equationz}
    \int_{\mathbb{R}} e(- v y^2) 
    W \Big ( \frac{y + \frac{\ell}{2 v}}{\sqrt{N}} \Big ) dy 
    =  \frac{e(-\tfrac {\text{sgn}(v)}{8})}{\sqrt{2 |v|}} 
    \int_{\mathbb{R}} e \Big ( \frac{\xi^2}{4 v N} \Big ) 
    e \Big ( \frac{\ell \xi}{2 v \sqrt{N}} \Big ) \widehat{W}(\xi) d \xi.
    \end{equation}
    Repeated integration by parts gives $|\widehat{W}(\xi)| \ll_{V} (1 + |\xi|^4)^{-1}$
    and therefore, 
    $$
     \int_{\mathbb{R}} \xi^2 |\widehat{W}(\xi)| d \xi
     \ll_{V} 1.
    $$
     Using the expansion,
     $$
     e \Big ( \frac{\xi^2}{4 v N} \Big ) = 1 + O \Big ( \frac{\xi^2}{4 |v| N} \Big )
     $$
     (valid uniformly in all parameters)
     we infer from \eqref{eq:equationz} that
\begin{equation} 
I(\ell,v) =  
e \Big (\frac{\ell^2}{4 v} - \frac{\text{sgn}(v)}{8} \Big ) 
\sqrt{\frac{2 N}{|v|}} 
\Big ( W \Big ( \frac{\ell}{2 v \sqrt{N}} \Big ) + O_{V}(|v|^{-1} N^{-1} ) \Big ) .
\end{equation} 
We have therefore obtained for $\ell \neq 0$,
$$
S(\ell) = \sqrt{2 N} 
\sum_{\substack{v \neq 0 \\ \text{sgn}(\ell / v) = 1}} 
\frac{e(- \tfrac{\mathrm{sgn}(v)}{8})}{\sqrt{|v|}} 
e \Big( \frac{\ell^2}{4 v} \Big)
W \Big ( \frac{\ell}{2 v \sqrt{N}} \Big ) + 
O_{V} (N^{-1/2}).
$$
In the case $\ell = 0$,
$$
S(0) = N \widehat{V}(0) + O_{V,A}(N^{-A}).
$$
\subsubsection{Sum over $\ell$}
For a non-zero integer $v$, we let
\begin{equation} \label{eq:start}
\Sigma(v) := \frac{1}{N} \sum_{\substack{\ell \in \mathbb{Z} \\
\mathrm{sgn}(\ell/v) = 1}} 
\widehat{\Phi} \Big ( \frac{\ell}{N} \Big ) e \Big ( \frac{\ell^2}{4 v} \Big ) 
e \Big ( - \frac{\ell a}{q} \Big )
e(\ell \theta) W \Big ( \frac{\ell}{2 v \sqrt{N}} \Big ).
\end{equation}
%\marginpar{Add coprimality assumptions}
Notice that $\Sigma(-v) = \overline{\Sigma(v)}$. It follows from this and our previous computation that, 
\begin{align} \nonumber
R_{\Phi, V} \Big ( \frac{a}{q} - \theta; N \Big ) 
& = \widehat{\Phi} (0) \widehat{V}(0)
+ \sqrt{2N} \sum_{v \neq 0}  e ( -\tfrac{\mathrm{sgn}(v)}{8} )
\frac{\Sigma(v)}{\sqrt{|v|}} 
+ O_{V}(N^{-1/2}) \\ \label{eq:masterz} & 
= \widehat{\Phi}(0) \widehat{V}(0)
+ 2 \sqrt{2 N} \cdot \mathrm{Re} 
\Big ( \sum_{v > 0} \frac{e( - \tfrac 18 )}
{\sqrt{v}} \Sigma(v) \Big ) + O_{V}(N^{-1/2})
\end{align}
so that we can assume $v > 0$.
%Write $[ 4 v , q] = 4 v q$,
%recalling $q^{\star} = q / d$, $d = (q,v)$,
%and $v^{\star} = v / d$.
Let $r := (v,q)$.
Lemma \ref{le:poisson} implies that $\Sigma(v)$ is equal to 
\begin{align*}
\frac{r}{4v q N} \sum_{u \in \mathbb{Z}} 
& \sum_{x \mmod{4v q/ r}} 
e \Big ( \frac{x^2}{4 v} - \frac{a x}{q} 
+ \frac{r u x}{4 v q}  \Big ) \int_{0}^{\infty} 
\widehat{\Phi} \Big ( \frac{y}{N} \Big ) 
 e ( \theta y) 
 W \Big ( \frac{y}{2v \sqrt{N}} \Big ) 
 e \Big ( - \frac{r u y}{4 v q} \Big ) dy.
\end{align*}
Notice that the integral is from $(0, \infty)$ rather than $(-\infty, \infty)$ because 
we sum over \textit{positive} $\ell$. In particular, to execute the sum only over positive $\ell$ we implicitly replace the weight function $W(x)$ by $W(x) \mathbf{1}_{x > 0}$. The latter is a smooth function because $W$ is compactly supported away from $0$.  
By a change of variables $y / (2 v \sqrt{N}) \mapsto y$, and recalling that $W(x) = x V(x^2)$, we see that 
$$\int_{0}^{\infty} 
\widehat{\Phi} \Big ( \frac{y}{N} \Big ) 
 e ( \theta y) 
 W \Big ( \frac{y}{2v \sqrt{N}} \Big ) 
 e \Big ( - \frac{r u y}{4 v q} \Big ) dy 
 = 2 v \sqrt{N} F_{\Phi, V, N\theta} 
\Big ( \frac{v}{\sqrt{N}}, \frac{r u}{q / \sqrt{N}} \Big ).$$
%\textbf{TODO: Need to consider the case of $d = (v,q)$ and use the assumption that $q$
%has large prime factors}
%We now split the sum over $v$ as
%$$
%\sum_{\substack{v > 0 \\ (v,q) = 1}} \frac{\Sigma(v)}{\sqrt{v}} + \sum_{\substack{v > 0 \\ (v,q) > 1}} \frac{\Sigma(v)}{\sqrt{v}}. 
%$$
%and analyze the two sums separately. 
%\subsubsection{Terms with $(v,q) = 1$}
We now analyze the exponential sum. 
Let $r = (v,q)$ and $v^{\star} = v / r$ and $q^{\star} = q / r$. 
For any non-zero 
and pair-wise coprime $i,j,h\in \mathbb{Z}$, 
the Chinese remainder theorem implies 
\begin{equation}\label{eq Chinese remainder}
\frac{1}{ij} \equiv 
\frac{\overline{i}}{j} + \frac{\overline{j}}{i} \mmod{1},
\quad \mathrm{and}\quad
\frac{1}{ijh} \equiv 
\frac{\overline{jh}}{i} + \frac{\overline{ih}}{j} +
\frac{\overline{ij}}{h} \mmod{1}.
\end{equation}
%Using the Chinese remainder theorem 
We write $x$ 
in terms of three new variables $e \mmod{q^{\star}}$, $f \mmod{4 v^{\star}}$ and $g \mmod{r}$, so that, 
$x \equiv e \mmod{q^{\star}}$, $x \equiv f \mmod{4v^{\star}}$ and $x \equiv g \mmod{r}$. Using \eqref{eq Chinese remainder}
we see that
\begin{align*}
  \frac{x^2}{4 v} - \frac{a x}{q} + \frac{u x}{4v q^{\star}} & \equiv \frac{x^2 \overline{r} + u x \overline{q^{\star} r}}{4 v^{\star}} + \frac{x (u \overline{4 v^{\star} r} - a \overline{r})}{q^{\star}} + \frac{x^2 \overline{4 v^{\star}} - a x \overline{q^{\star}} + u x \overline{4 v^{\star} q^{\star}}}{r}  \mmod{1}
  \\ & \equiv \frac{f^2 \overline{r} + u f \overline{q^{\star} r}}{4 v^{\star}} + \frac{e (u \overline{4 v^{\star} r} - a \overline{r})}{q^{\star}} + \frac{g^2 \overline{4 v^{\star}} - a g \overline{q^{\star}} + u g \overline{4 v^{\star} q^{\star}}}{r} \mmod{1}.
\end{align*}
Executing the sum over $e$ we obtain the condition $4 v^{\star} a \equiv u \mmod {q^{\star}}$
times $q^{\star}$. On the other hand the sum over $f$ and $g$ gives rise to Gauss sums. We will need an exact evaluation 
only in the case $r = 1$, for $r > 1$ a bound will suffice. Thus, if $r = 1$, 
$$
\sum_{x \mmod{4v q / r}} 
e \Big ( \frac{x^2}{4 v} - \frac{a x}{q} 
+ \frac{r u x}{4 v q}  \Big ) = q \mathbf{1} \Big ( 4 v a \equiv u \mmod{q} \Big ) \cdot \sqrt{2} e \Big ( \frac{1}{8} \Big ) \mathbf{1}_{2 \mathbb{Z}}(u) \sqrt{4 v} e \Big ( - \frac{(u / 2)^2 \overline{q}^2}{4 v} \Big ), 
$$
while if $r > 1$, then, 
$$
\sum_{x \mmod{4v q / r}} 
e \Big ( \frac{x^2}{4 v} - \frac{a x}{q} 
+ \frac{r u x}{4 v q}  \Big ) \ll \frac{q}{r} \cdot \mathbf{1} \Big ( 4 v^{\star} a \equiv u \mmod{q^{\star}} \Big ) \sqrt{v}.
$$
As a result we get for $r = 1$, (after the change of variable $u \mapsto 2 u$)
\begin{equation} \label{eq:slavez}
\Sigma(v) = \sqrt{2} e \Big ( \frac 18 \Big )\sqrt{\frac{v}{N}}
 \sum_{\substack{u \in \mathbb{Z}
 \\ 2 v a \equiv u \mmod{q}}} 
 e \Big ( - \frac{u^2 \overline{q}^2}{4 v}  \Big ) 
 F_{\Phi, V, N \theta} 
 \Big ( \frac{v}{\sqrt{N}}, \frac{u}{q / \sqrt{N}} \Big ).
\end{equation}
Meanwhile for $r  > 1$ we have, 
$$
\Sigma(v) \ll \sqrt{\frac{v}{N}} \sum_{\substack{u \in \mathbb{Z}, u \neq 0 \\ 4 v^{\star} a \equiv u \mmod{q^{\star}}}} \Big | F_{\Phi, V, N \theta} \Big ( \frac{v}{\sqrt{N}}, \frac{r u}{q / \sqrt{N}} \Big ) \Big |. 
$$
Notice that the condition $u \neq 0$ in the above sum follows from $4 v^{\star} a \equiv u \mmod{q^{\star}}$ and $(4 v^{\star} a, q^{\star}) = 1$.
Therefore, we can write, 
\begin{align*}
R_{\Phi, V} & \Big ( \frac{a}{q} - \theta; N \Big ) = \widehat{\Phi}(0) \widehat{V}(0) \\ & + 4 \mathrm{Re} \Big ( \sum_{\substack{v > 0,\ u \in \mathbb{Z}, \ (v, q) = 1 \\ 2 v a \equiv u \pmod{q}}} e \Big ( - \frac{u^2 \overline{q}^2}{4 v} \Big ) F_{\Phi, V, N \theta} \Big ( \frac{v}{\sqrt{N}}, \frac{u}{q / \sqrt{N}} \Big )\Big ) \\ & \ \ \ \ \ \ + O \Big ( \sqrt{N} \sum_{\substack{v > 0 \\ (v,q) > 1}} \frac{|\Sigma(v)|}{\sqrt{v}} \Big ) + O_V(N^{-1/2}).
\end{align*}
\subsubsection{The terms with $r := (v,q) > 1$}
We will now show that the contribution of the terms with $r > 1$ is negligible. 
In the following lines the value of the large constant $A > 10$ may change from occurence to occurence. 
Splitting into dy-adic intervals, we get, 
$$
\sqrt{N} \sum_{\substack{v > 0 \\ (v,q) > 1}} \frac{|\Sigma(v)|}{\sqrt{v}} \ll \sum_{\substack{r | q \\ r > 1}} \sum_{K, L \geq 0} \ \sum_{\substack{|v| \in I_K \sqrt{N}, \ r | v \\ 0 < |u| \in I_L, \ (v,q / r) = 1 \\ 4 v a \equiv r u \pmod{q / r}}} \Big | F_{\Phi, V, N \theta} \Big ( \frac{v}{\sqrt{N}}, \frac{r u}{q / \sqrt{N}} \Big ) \Big |   
$$
where $I_K := [2^K - 1, 2^{K + 1} -1]$. We notice that the terms with $2^L \geq \max ((2^K \Delta)^{2},\Delta^4 r^2)$ contribute, 
$$
\ll_{A, \Phi, V} \sum_{r \geq 1} \sum_{\substack{K \geq 0 \\ 2^L \geq (2^K \Delta)^{2}}} \Big ( \frac{2^K \sqrt{N}}{q} + 1 \Big ) 2^{L} \cdot \frac{1 + 2^{A K}}{1 + (r 2^L / \Delta)^A} \ll_{\varepsilon, A, \Phi, V} \Delta^{-A}
$$
Furthermore, the term with $2^K > \Delta r$ and $2^L \leq \max((2^{K} \Delta)^2, \Delta^4 r^2) \leq (2^K \Delta)^2$ contribute, 
$$
\ll_{A, \Phi, V} \sum_{r \geq 1} \sum_{\substack{K \geq 0 \\ 2^K \geq \Delta r}} \Big ( \frac{2^K \sqrt{N}}{q} + 1 \Big ) \cdot (2^K \Delta)^{2} \cdot \frac{1}{1 + 2^{K A}}  \ll_{A,\Phi, V} \Delta^{-A}.
$$
Thus we are left with terms with $2^K \leq \Delta r$ and $2^L \leq \Delta^4 r^2$. Since all the prime factors of $q$ are $\geq \Delta^{2001}$ if $r > 1$ and $r | q$ then $r \geq \Delta^{2001}$. Thus, 
$$
F_{\Phi, V, N \theta} \Big ( \frac{v}{\sqrt{N}}, \frac{r u}{q / \sqrt{N}} \Big ) \ll_{A, \Phi, V} \frac{1 + \Delta^A}{1 + (r / \Delta)^A} \ll \frac{1 + \Delta^{2 A}}{1 + r^{A / 2} \Delta^{1000 A}} \ll r^{-A / 2} \Delta^{-A}.
$$
As a result the bound in this case is
$$
\ll_{A, \Phi, V} \Delta^{-A} \sum_{r \geq 1} \frac{1}{r^A} \cdot \Delta^5 r^3  \ll_{A, \Phi, V} \Delta^{-A}.
$$
%\subsubsection{The terms with $(u,v) \not \in \mathcal{B}_{a, \Delta, q}$}

%Reprising the previous argument, we can restrict to the region $|v| \leq \Delta \sqrt{N}$ and $|u| \leq \Delta^4$ at the cost of an error that is $\ll_{A, \Phi, V} \Delta^{-A}$. Furthermore, we can then extend the $v$ sum back to include all terms at the cost of an error term that is $\ll_{A, \Phi} \Delta^{-A}$. 

\subsubsection{Conclusion} Combining \eqref{eq:masterz} and \eqref{eq:slavez},
while keeping track of the error terms,
we infer
\begin{align*}
R_{\Phi, V} \Big ( \frac{a}{q} - \theta; N \Big ) 
= \widehat{\Phi}(0) \widehat{V}(0) + 4  \mathrm{Re} \Big (  
\sum_{\substack{v > 0, \ (v,q) = 1 \\ u \in \mathbb{Z} \\ 2 v a \equiv u \mmod{q}}}   
e \Big (  \frac{- \overline{q}^2 u^2}{4 v} \Big ) &
 F_{\Phi, V, N \theta} 
 \Big ( \frac{v}{\sqrt{N}}, 
 \frac{u}{q / \sqrt{N}} \Big ) \Big ) \\ & 
 + O_{A, \Phi, V}( \Delta^{-A} + N^{-1/2} ).
\end{align*}
Notice that the conjugate of the sum over $u$ is equal to the same sum with $v$ replaced by $-v$. This follows from $\overline{F_{\Phi, V, N \theta}}(\xi, \eta) = F_{\Phi, V, N \theta}( -\xi, -\eta)$ and the fact that in the sum over $u \in \mathbb{Z}$ we can replace $u$ by $-u$. As a result, 
\begin{align*}
R_{\Phi, V} \Big ( \frac{a}{q} - \theta; N \Big ) 
= \widehat{\Phi}(0)\widehat{V}(0) + 
2 \sum_{\substack{
 v \in \mathbb{Z}, (v,q) = 1 \\ u \in \mathbb{Z} \\ 2 v a \equiv u \mmod{q}}}  
e \Big (  \frac{- \overline{q}^2 u^2}{4 v} \Big ) &
 F_{\Phi, V,N \theta} \Big ( \frac{v}{\sqrt{N}}, \frac{u}{q / \sqrt{N}} \Big ) \\ & + O_{A, \Phi, V}(\Delta^{-A} + N^{-1/2}).
\end{align*}

\subsubsection{Restricting to $(u,v) \in \mathcal{B}_{a, \Delta, q}$}

Reprising the argument that we used to restrict to $(v,q) = 1$ we can restrict the sum over $u$ and $v$ to $|v| \leq \Delta \sqrt{N}$ and $|u| \leq \Delta^4$ at the cost of an error that is $\ll_{A, \Phi, V} \Delta^{-A}$. This immediately gives a point-wise bound of $\ll_{\Phi, V} \Delta^4$. Furthermore, we can then extend the $v$ sum back to include all terms at the cost of an error term that is $\ll_{A, \Phi} \Delta^{-A}$. This gives the restriction to $(u,v) \in \mathcal{B}_{a, \Delta, q}$. 
With this restriction in place, we can insert into the phase the term 
$$
e \Big ( \frac{u^2}{4 v q^2} \Big ) = 1 + O \Big ( \frac{\Delta^6}{v N} \Big )
$$
without affecting any of the error terms.

%  \end{proof}

%  We can now deduce Proposition \ref{prop3}. Indeed since $q$ has no prime factors $\leq \Delta^2$ it follows that either $d = 1$ or $d > \Delta^2$. If $d > \Delta^2$ then the contribution of the sum over $u$ is $\ll_{A} \Delta^{-A}$. 
%  Therefore we can re-write the main term as
%  $$
%\sum_{\substack{v,u \neq 0 \\ 2 v a \equiv u \mmod{q} \\ 
%(v, q) = 1}}  e \Big (  \frac{- \overline{q^{\star}}^2 u^2}{4 v} 
%+ \frac{(u+av^{\star}) a \overline{q}^2}{d} \Big ) F_{\Phi, V, \theta} \Big ( \frac{v}{\sqrt{N}}, \frac{u}{q / \sqrt{N}} \Big )
% $$
% plus a main term that is 
% $\ll_{A} \Delta^{-A} + N^{-1/2} \| V \|_{C^2}$. 
% Notice that we can add now the restriction 
% $1 \leq |u| \leq \Delta^{1 + \varepsilon}$ at the price 
% of an additional error $\ll_{A, \varepsilon} \Delta^{-A}$. 
% The conditions $v \neq 0, u \neq 0$ and $(v, q) = 1$ 
% are now redundant: indeed, since 
% $1 \leq |u| \leq \Delta^{1 + \varepsilon}$ 
% and all prime factors of $q$ are larger than 
% $\Delta^2$ we have $(u, q) = 1$; 
% and since $2 v a \equiv u \mmod{q}$ 
% it follows that also $(v, q) = 1$, hence also $v \neq 0$. 

 \section{Proof of Proposition \ref{prop4}}

 Throughout it will be convenient to use vector notation writing $\boldsymbol{v}$ for a vector in $\mathbb{R}^k$ with component $(v_1, \ldots, v_k)$,
 and the abbreviation 
 $$
 L := \sum_{q \in \mathcal{Q}_{\Delta, N}} \varphi(q).
 $$
 Proving Proposition \ref{prop4} 
 amounts to establishing a power-saving asymptotic for
  $$
\frac{N}{L} \sum_{\substack{(a,q) = 1
\\ q \in \mathcal{Q}_{\Delta,N}}} \int_{\mathbb{R}} \Big ( 
\sum_{(u,v)\in\mathcal{B}_{a, \Delta, q}} e \Big (  - \frac{\overline{q}^2 u^2}{4 v} + \frac{u^2}{4 v q^2} \Big ) F_{\Phi, V, N \theta} \Big ( \frac{v}{\sqrt{N}}, \frac{u}{q / \sqrt{N}} \Big ) \Big )^k \phi(N \theta) d \theta.
 $$
%Before continuing in earnest,
%we record some useful observations.
%In particular, we record an "a priori" 
%bound on 
%$M_k$ that will play a role later-on
%where it is crucial to
%control how fast $k\mapsto M_k$
%may grow.
Expanding the $k^{th}$-power 
and making a linear change of variable in $\theta$, we get
 \begin{equation}\label{eq: kth moment expanded}
\frac{1}{L} 
\sum_{\substack{q \in 
\mathcal{Q}_{\Delta, N}\\ (a,q) = 1 \\ (\mathbf{u}, \mathbf{v}) \in \mathcal{B}_{a, \Delta, q}^k}}
%\\ 1 \leq |u_i| \leq \Delta^{1 + 2\nu}
%\\v_i\in \mathbb{Z}, (v_i,q)=1 
%\\ 2 v_i a \equiv u_i \mmod{q}
e \Big ( -\sum_{i \leq k} 
\frac{\overline{q}^2 u_i^2}{4 v_i} + \sum_{i \leq k} \frac{u_i^2}{4 v_i q^2}
\Big ) 
\int_{\mathbb{R}} 
\prod_{i = 1}^{k} 
F_{\Phi, V, \theta} 
\Big ( \frac{v_i}{\sqrt{N}}, 
\frac{u_i}{q / \sqrt{N}} \Big ) 
\phi(\theta) d \theta. 
 \end{equation}
The variables $v_i$ and $u_i$ are chained by the condition $v_i u_j \equiv v_j u_i \mmod{q}$ for all $1 \leq i, j \leq k$. 
It therefore makes sense 
to introduce variables 
$\ell_i \in \mathbb{Z}$ such that,
 $$
 u_1 v_i = u_i v_1 + q \ell_i
 $$
 for all $1 \leq i \leq k$.  %Notice that $|\ell_i| \leq \Delta^{2 \varepsilon}$ since $|u_i| \leq \Delta^{1 + 2\nu}$, $|v_i| \leq \sqrt{N} \Delta^{\varepsilon}$ and $q \asymp \Delta \sqrt{N}$. 
 To proceed further 
 we need to diminish the modulus 
 of the phase function 
 in \eqref{eq: kth moment expanded}. 
 This is accomplished in the next lemma.
\begin{lemma} \label{le:transform}
  Let $q \neq 0$, odd,  be given. Let $\{u_i\}_i, \{v_i\}_i, 
  \{\ell_i\}_i$ be sequences of integers 
such that $u_1 v_i = u_i v_1 + q \ell_i$ 
and $(4 u_i v_i,q)=1$ for all $1 \leq i \leq k$.
  Then, 
  $$
  \sum_{i \leq k}
  \frac{\overline{q}^2 u_i^2}{4 v_i} \equiv \frac{u_1 \overline{4 v_1} 
  \sum_{1 \leq i \leq k} u_i}{q^2} - \frac{u_1 \overline{4 v_1^2} 
  \sum_{1 \leq i \leq k} \ell_i}{q} 
  + \sum_{i \leq k} \frac{u_i^2}{4 v_i q^2} \mmod{1}.
  $$
  \end{lemma}
\begin{proof}
Let $w := 4 v_1 \ldots v_k$. 
To avoid confusion, we denote the modular inverse of $x$ with respect
to modulus $q$ by $\overline{x}_{q}$.
Thus, 
$$
\sum_{i \leq k} \frac{(\overline{q})_{4 v_i}^2 u_i^2}{4 v_i} 
\equiv \sum_{i \leq k} \frac{\overline{q}_{w}^2 u_i^2}{4 v_i} 
\equiv \frac{a {\overline{q}}_{w}^2}{4 v_1 \ldots v_k} 
\,\mathrm{mod}\,1 \quad \mathrm{with} \quad
a = w \sum_{i \leq k} \frac{u_i^2}{4 v_i} \in \mathbb{Z}.
$$
By the Chinese remainder theorem,
\begin{equation} \label{crt1}
\frac{a \overline{q}_{w}^2}{w} \equiv - 
\frac{a \overline{w}_{q^2}}{q^2} + \frac{a}{w q^2}
\, \,\mathrm{mod}\,1
\end{equation}
so it remains to understands
\begin{equation} \label{crt2}
a \overline{w}_{q^2} \equiv \sum_{i \leq k} u_i^2 (\overline{4 v_i})_{q^2}
\,\mathrm{mod}\,\,q^2.
\end{equation}
Using that $u_1 v_i = u_i v_1 + q \ell_i$ and that
\begin{align*}
  (\overline{4 v_i})_{q^2} & \equiv u_1 (\overline{4 u_1 v_i})_{q^2} \mmod{q^2}
  \\ & \equiv u_1 (\overline{4 u_i v_1 + 4 q \ell_i})_{q^2} \mmod{q^2}
  \\ & \equiv u_1 (\overline{4 u_i v_1})_{q^2} \cdot (\overline{1 + q \ell_i (\overline{u_i v_1})_{q^2}})_{q^2} \mmod{q^2}
  \\ & \equiv u_1 (\overline{4 u_i v_1})_{q^2} \cdot (1 - q \ell_i (\overline{u_i v_1})_{q^2})  \mmod{q^2}
  \\ & \equiv u_1 (\overline{4 u_i v_1})_{q^2} - q \ell_i u_1 (\overline{4 u_i^2 v_1^2})_{q^2} \mmod{q^2}. 
\end{align*}
Therefore
$$
u_i^2 (\overline{4 v_i})_{q^2} \equiv u_1 u_i (\overline{4 v_1})_{q^2} - q \ell_i u_1 (\overline{4 v_1^2})_{q^2} \mmod{q^2}. 
$$
and thus
\begin{equation} \label{crt3}
\sum_{1 \leq i \leq k} u_i^2 (\overline{4 v_i})_{q^2} \equiv u_1 (\overline{4 v_1})_{q^2} \sum_{1 \leq i \leq k} u_i - q u_1 (\overline{4 v_1^2})_{q^2} \sum_{1 \leq i \leq k} \ell_i \mmod{q^2}. 
\end{equation}
Combining Equations \eqref{crt1}, \eqref{crt2}, \eqref{crt3} we get,
$$
\frac{a \overline{q}_{w}^2}{w} \equiv - \frac{u_1 (\overline{4 v_1})_{q^2} \sum_{i = 1}^{k} u_i}{q^2} + \frac{u_1 (\overline{4 v_1^2})_{q} \sum_{i = 1}^{k} \ell_i}{q} + \frac{a}{w q^2} \mmod{1},
$$
as required.
\end{proof}
In the notation of Lemma \ref{le:transform},
letting $ \kappa:= u_1 + \ldots + u_k$
and $\eta:= \ell_1 + \ldots + \ell_k$,
$$
  e\Big(- \sum_{i \leq k}
  \frac{\overline{q}^2 u_i^2}{4 v_i}
  + \sum_{i \leq k} \frac{u_i^2}{4 v_i q^2} \Big)
  =  
  e\Big(- \frac{u_1 \overline{4 v_1} 
  \kappa }{q^2} + 
  \frac{u_1 \overline{4 v_1^2} \eta}{q} 
  \Big).
$$
This motivates us to separate the terms 
$u_i$ and $v_i$ into two categories; 
\begin{enumerate}
\item [a.] the main term, corresponding to terms with $\kappa = 0 = \eta$
\item [b.] the error term, corresponding to terms with $\kappa \neq 0$ or $\eta \neq 0$
\end{enumerate}
We denote the main term by, 
\begin{equation} \label{eq:maintt}
\mathcal{M}_k:=\frac{1}{L} \int_{\mathbb{R}} 
\sum_{\substack{q \in \mathcal{Q}_{\Delta, N} \\ (a,q) = 1}} \sum_{\substack{(\mathbf{u}, \mathbf{v}) \in \mathcal{B}_{a, \Delta, q}^{k} 
\\ v_1 + \ldots + v_k = 0 \\ u_1 + \ldots + u_k = 0}} 
G_{\Phi, V, \theta, k}
\Big ( \frac{\boldsymbol{v}}{\sqrt{N}}, 
\frac{\boldsymbol{u}}{q / \sqrt{N}} \Big )\phi(\theta) d \theta
\end{equation}
where the function $G_{\Phi, V, \theta, k}$ 
is defined as
$$
G_{\Phi, V, \theta, k}(\boldsymbol{\xi},\boldsymbol{\eta}) = \prod_{i = 1}^{k} F_{\Phi, V, \theta} ( \xi_i, \eta_i).
%e \Big ( \frac{\eta_i^2}{4 \xi_i N^{3/2}} \Big ).
$$
% \Niclas{Deleted an incorrect factor 
% of $e \Big ( \frac{\eta_i^2}{4 \xi_i N^{3/2}} \Big )$ in the definition 
% of  $G$.}
We will now asymptotically estimate $\mathcal{M}_k$ and then show that
the contribution of the remaining terms is negligible. 

\subsection{The main term}
In evaluating the main term $\mathcal{M}_k$ it's enough for us to focus on the sum over $v_1, \ldots, v_k$ and the sum over $(a,q) = 1$ and $q \in \mathcal{Q}_{\Delta, N}$. We leave the variables $u_1, \ldots, u_k$ and $\theta$ untouched. 
To capture the condition $v_1 + \ldots + v_k = 0$ we introduce a smooth function $w\geq 0$ compactly supported in $(-1/2, 1/2)$ 
and such that $w(0) = 1$. Set $w(\boldsymbol{v}) = w(v_1 + \ldots + v_k)$. We notice that $u_1 (v_1 + \ldots + v_k) = q (\ell_1 + \ldots + \ell_k)$ and therefore if $v_1 + \ldots + v_k \neq 0$ then $|v_1 + \ldots + v_k| \geq q / |u_1| > \sqrt{N} / \Delta^3$. 
Therefore, 
$$
w \Big ( \frac{\boldsymbol{v}}{\sqrt{N} / \Delta^3} \Big ) \neq 0
$$
if and only if $v_1 + \ldots + v_k = 0$.  
Executing Poisson summation over $v_1, \ldots, v_{k}$ we obtain
\begin{align*}
& \sum_{\substack{(\boldsymbol{v}, \boldsymbol{u}) \in \mathcal{B}_{a, \Delta, q}^{k} \\ v_1 + \ldots + v_k = 0}}  G_{\Phi, V, \theta, k} \Big ( \frac{\boldsymbol{v}}{\sqrt{N}}, \frac{\boldsymbol{u}}{q / \sqrt{N}} \Big ) = \sum_{\substack{(\boldsymbol{v}, \boldsymbol{u}) \in \mathcal{B}_{a, \Delta, q}^{k}}}  G_{\Phi, V, \theta, k} \Big ( \frac{\boldsymbol{v}}{\sqrt{N}}, \frac{\boldsymbol{u}}{q / \sqrt{N}} \Big ) w \Big (\frac{\boldsymbol{v}}{\sqrt{N} / \Delta^3} \Big ) \\ & = \frac{1}{(q / \sqrt{N})^{k}} \sum_{\substack{\boldsymbol{t} \in \mathbb{Z}^{k}}} e \Big ( - \frac{\overline{2a} \langle \boldsymbol{u}, \boldsymbol{t} \rangle }{q} \Big ) \int_{\substack{\boldsymbol{\xi} \in \mathbb{R}^{k}}} G_{\Phi, V, \theta, k} \Big ( \boldsymbol{\xi}, \frac{\boldsymbol{u}}{q / \sqrt{N}} \Big ) e \Big ( - \frac{\langle \boldsymbol{\xi}, \boldsymbol{t} \rangle}{q / \sqrt{N}} \Big ) w (\Delta^3 \boldsymbol{\xi} ) d \boldsymbol{\xi}
\end{align*}
where $\mathbf{t} := (t_1, \ldots, t_k)$. 
Since, 
$$\Big ( \prod_{i = 1}^{k} \frac{\partial^{\ell_i}}{\partial^{\ell_i} \xi_i} \Big ) G_{\Phi, V, \theta, k} ( \boldsymbol{\xi}, \boldsymbol{\eta} ) w(\Delta^3 \boldsymbol{\xi}) \ll_{\Phi, V, \boldsymbol{\ell}} \prod_{i = 1}^{k} \frac{\Delta^{3 \ell_i}}{1 + |\xi_i|^{2}}$$ with $\boldsymbol{\ell} = (\ell_1, \ldots, \ell_k)$ 
we can integrate by parts in each variable $\xi_i$, getting that, the integral over $\boldsymbol{\xi} \in \mathbb{R}^k$ is bounded by 
$$
\ll_{\Phi, V, \boldsymbol{\ell}} \prod_{i = 1}^k \frac{\Delta^{3 \ell_i}}{1 + |t_i / \Delta|^{\ell_i}}.
$$
This allows us to introduce the truncation $|t_i| \leq N^{1/(100000k)} \leq q^{1/(50000k)}$ at the price of
a total error term that $\ll_{\Phi, V, A} N^{-A}$ for any given $A > 10$. 

We now execute the sum over $(a,q)  = 1$. A main term $\varphi(q)$ comes out from terms with $\langle \boldsymbol{t}, \boldsymbol{u} \rangle = 0$. 
In all other cases, since $|\langle \boldsymbol{t}, \boldsymbol{u} \rangle | \leq k \Delta^4 q^{1/(50000k)}$ we have, 
$$
\sum_{(a,q) = 1} e \Big ( - \frac{\overline{2 a} \langle \boldsymbol{t}, \boldsymbol{u} \rangle}{q}  \Big ) = \sum_{\substack{d | q \\ d | \langle \boldsymbol{t}, \boldsymbol{u} \rangle}} d \mu \Big ( \frac{q}{d} \Big ) \ll |\langle \boldsymbol{t}, \boldsymbol{u} \rangle|^2 \ll k^2 \Delta^8 q^{1/(25000k)}.
$$
Finally it remains to execute the sum over $q \in \mathcal{Q}_{\Delta, N}$. 
We let
\begin{align*}
\mathfrak{F}_{\boldsymbol{t}, \boldsymbol{u}, \theta}(x) & := \frac{1}{x^{k}}  \int_{\substack{\boldsymbol{\xi} \in \mathbb{R}^{k}}} G_{\Phi, V, \theta, k} \Big ( \boldsymbol{\xi}, \frac{\boldsymbol{u}}{x} \Big ) e \Big ( - \frac{\langle \boldsymbol{\xi}, \boldsymbol{t} \rangle}{x} \Big ) w(\Delta^3 \boldsymbol{\xi}) d \boldsymbol{\xi} \\ & = \int_{\boldsymbol{\xi} \in \mathbb{R}^k} G_{\Phi, V, \theta, k} \Big ( x \cdot \boldsymbol{\xi}, \frac{\boldsymbol{u}}{x} \Big ) e (- \langle \boldsymbol{\xi}, \boldsymbol{t} \rangle) w (x \cdot \Delta^3 \boldsymbol{\xi}) d \boldsymbol{\xi}.
\end{align*}
Let $\widetilde{L} := \sum_{q \in \mathcal{Q}_{\Delta, N}} \frac{\varphi(q)}{q}$.
By integration by parts, 
\begin{align} \nonumber
& \frac{1}{\widetilde{L}}  \sum_{q \in \mathcal{Q}_{\Delta, N}} \frac{\varphi(q)}{q} \cdot \Big ( \frac{q}{\sqrt{N}} \mathfrak{F}_{\boldsymbol{t}, \boldsymbol{u}, \theta} \Big ( \frac{q}{\sqrt{N}} \Big ) \Big ) - \frac{1}{\Delta} \int_{\Delta}^{2 \Delta} v \mathfrak{F}_{\boldsymbol{t}, \boldsymbol{u}}(v) dv  \\ \nonumber & = \int_{\Delta}^{2 \Delta} v \mathfrak{F}_{\boldsymbol{t}, \boldsymbol{u}, \theta}(v) d \Big ( \frac{1}{\widetilde{L}} \sum_{q \in [\Delta \sqrt{N},v \sqrt{N}] \cap \mathcal{Q}_{\Delta, N}} \frac{\varphi(q)}{q} - \frac{v - \Delta}{\Delta} \Big ) \ll_{\Phi, V, k} Q^{-1/2 + 1/1000}
\end{align}
and where we used Lemma \ref{le:comput} and the estimate, 
$$
\frac{\partial}{\partial u} \mathfrak{F}_{\boldsymbol{t}, \boldsymbol{u}}(u) \ll_{\Phi, V, k} \Delta^{10} \ll_{\Phi,V, k} \log Q.
$$
Therefore, 
$$
\frac{1}{L} \sum_{q \in \mathcal{Q}_{\Delta, N}} \varphi(q) \mathfrak{F}_{\boldsymbol{t}, \boldsymbol{u}, \theta} \Big ( \frac{q}{\sqrt{N}} \Big ) = \frac{\widetilde{L} \sqrt{N}}{L} \cdot \frac{1}{\Delta} \int_{\Delta}^{2 \Delta} u \mathfrak{F}_{\boldsymbol{t}, \boldsymbol{u}, \theta} (u) du + O_{\Phi, V, k}(Q^{-1/2 + 1/1000})
$$
Using again Lemma \ref{le:comput} and the bound $\mathfrak{F}_{\boldsymbol{t}, \boldsymbol{u}, \theta}(v) \ll v^{-k}$ we conclude that the above is 
$$
\frac{2}{3 \Delta^2} \int_{\Delta}^{2 \Delta} v \mathfrak{F}_{\boldsymbol{t}, \boldsymbol{u}, \theta}(v) dv + O_{\Phi, V, k}(Q^{-1/2 + 1/1000}). 
$$
Thus we conclude, 
$$
\mathcal{M}_{k} = \sum_{\substack{\boldsymbol{t}, \boldsymbol{u} \in \mathbb{Z}^k \\ \forall i : 1 \leq |u_i| \leq \Delta^4 \\ \langle \boldsymbol{t}, \boldsymbol{u} \rangle = 0 = \langle \boldsymbol{1}, \boldsymbol{u} \rangle }} \int_{\mathbb{R}} \Big ( \frac{2}{3 \Delta^2} \int_{\Delta}^{2 \Delta} v \mathfrak{F}_{\boldsymbol{t}, \boldsymbol{u}, \theta} (v) dv \Big )  \phi(\theta) d \theta   + O_{\Phi, V, k}(\Delta^{4k} Q^{-1/2 + 1/1000}).
$$

\subsection{The error term}

We now consider the contribution of terms with 
$\kappa := u_1 + \ldots + u_k \neq 0$ or with 
$\eta := \ell_1 + \ldots + \ell_k \neq 0$. 
Executing the summation over $(a,q) = 1$ 
and writing $u_1 v_i = u_i v_1 + q \ell_i$ we express 
the error term as,
 \begin{equation}\label{eq: error term first form}
\frac{1}{L} 
\sum_{\substack{q \in 
\mathcal{Q}_{\Delta, N} 
\\ 1 \leq |u_i| \leq 
\Delta^{4} 
\\ (v_1 , q) = 1
\\ u_1 v_i = u_i v_1 + q \ell_i 
\\ (\kappa , \eta) \neq (0, 0)}} 
\varphi(q)
e \Big ( - \frac{u_1 \overline{4 v_1} 
\kappa}{q^2} + \frac{u_1 
\overline{4 v_1^2} \eta}{q} \Big ) 
\int_{\mathbb{R}} \prod_{i = 1}^{k}
F_{\Phi, V, \theta} \Big ( \frac{v_i}{\sqrt{N}}, \frac{u_i}{q / \sqrt{N}} \Big ) \phi(\theta) 
d \theta. 
 \end{equation}
 For any given $\varepsilon > 0$ using the decay of $F_{\Phi, V, \theta}$ we can 
 truncate the $v_i$ sum at $|v_i| \leq N^{1/2 + \varepsilon}$, incurring a total error of $\ll_{\Phi, V, A, \varepsilon} N^{-A}$ for any given $A$.
 In particular, this restricts $|\ell_i|$ to be less than $\ll_{\varepsilon} N^{\varepsilon}$.  
 We can therefore view the variables $u_i$ and $\ell_i$ as fixed, summing over them incurs only an additional error of $\ll \Delta^{4 k} N^{\varepsilon k}$. Furthermore, the sum is zero unless
 $u_1 | u_i v_1 + q \ell_i$ for all $1 \leq i \leq k$. In that case we see that all the variables $v_i$ are determined by $v_1$, 
 $$
 v_i = \frac{u_i v_1 + q \ell_i}{u_1}.
 $$
 Given a choice of $u_i, \ell_i$ we can therefore assume from now on that $v_1$ lies in one of the admissible (i.e $u_i v_1 + q \ell_i \equiv 0 \pmod{u_1}$ for all $i$) classes modulo $u_1$. Notice that this is allowed since the variables $u_i, \ell_i$ are considered fix (at the cost of an additional error $\ll \Delta^{4k } N^{\varepsilon k}$). 

Thus the problem is reduced to establishing a non-trivial power-saving bound for 
\begin{equation} \label{eq:mainexp}
\sum_{\substack{(v, q) = 1 \\ v \equiv \tau \mmod{u_1}}} e \Big ( - \frac{u_1 \overline{4 v_1} \kappa}{q^2} + \frac{u_1 \overline{4 v_1^2} \eta}{q} \Big ) G \Big ( \frac{v}{\sqrt{N}} \Big ),
\end{equation}
where at least one of $\kappa \neq 0$ or $\eta \neq 0$, $\tau$ is fixed, and
\begin{equation} \label{eq:expo}
G \Big ( \frac{v}{\sqrt{N}} \Big ) = \int_{\mathbb{R}} \prod_{i = 1}^{k} F_{\Phi, V, \theta} \Big ( \frac{u_i v + q \ell_i}{u_1 \sqrt{N}} , \frac{u_i}{q / \sqrt{N}} \Big ) \phi(\theta) d \theta.  
\end{equation}
We notice that $G$ remains a fairly inert function, since
$$
\frac{\partial^k}{\partial^k v} G(v) \ll_{\Phi, V,k} \Delta^{4 k}. 
$$ 
First we show that we can assume $\kappa \neq 0$: Indeed in the case $\kappa = 0$ it's enough to apply Poisson summation to modulus $q u_1$ to see that \eqref{eq:expo} is 
$$
\ll_{\varepsilon} N^{\varepsilon} \cdot \Big ( 1 + \sup_{|\ell| \leq N^{\varepsilon}} \Big | \sum_{\substack{(x, q) = 1 \\ x \equiv \tau \pmod{u}}} e \Big ( \frac{\eta \overline{4 x^2}}{q} - \frac{\ell x}{q u} \Big ) \Big | \Big ) \ll_{\varepsilon} N^{1/4 + \varepsilon}
$$
using the Chinese remainder theorem and the estimation of quadratic Gauss sums. 

We assume now that $\kappa \neq 0$. Recall that $q$ factors as $a \cdot b$ by assumption, with $Q^{1/1000} \leq a \leq 2 Q^{1/1000}$. We now perform Weyl differencing with shift $a^2 h u_1$. This allows us to re-write \eqref{eq:mainexp} as
$$
\sum_{\substack{v \equiv \tau \mmod{u_1}}} \frac{1}{H} \sum_{\substack{|h| \leq H \\ (v + h a^2 u_1, q) = 1}} e \Big ( \frac{\kappa \overline{4 (v + h a^2 u_1)}}{(a b)^2} - \frac{\eta \overline{4 (v + h a^2 u_1)^2}}{a b} \Big ) 
G \Big ( \frac{v}{\sqrt{N}} \Big ) + O_{\Phi, V, \varepsilon}(H N^{\varepsilon}),
$$
with $H = Q^{1/2000}$. 
Furthermore, 
$$
\frac{\kappa \overline{4 (v + h a^2 u_1)}}{(a b)^2} - \frac{\eta \overline{4 (v + h a^2 u_1)^2}}{a b} \equiv \frac{\kappa \overline{4 b^2 v}}{a^2} + \frac{\kappa \overline{4 a^2 (v + h a^2 u_1)}}{b^2} - \frac{\eta \overline{4 v^2 b}}{a} - \frac{\eta \overline{4 a (v + h a^2 u_1)^2}}{b}
\mmod 1
$$
and the condition $(v + h a^2 u_1, q) = 1$ can be expressed as $(v, a) = 1$ and $(v + h a^2 u_1, b) = 1$. 
We now apply the Cauchy-Schwarz 
inequality
bounding the above by
$$\ll_{\Psi, V,k } \Delta^2 N^{1/4} \mathcal{S}^{1/2} 
$$
where
$$
\mathcal{S} := \sum_{v \in \mathbb{Z}} \Big | \frac{1}{H} \sum_{\substack{|h| \leq H \\ (v + h a^2 u_1, b) = 1}} e \Big ( \frac{\kappa \overline{4 a^2 (v + h a^2 u_1)}}{b^2} - \frac{\eta \overline{4 a (v + h a^2 u_1)^2}}{b} \Big ) \Big |^2 W \Big ( \frac{v}{\Delta^4 \sqrt{N}} \Big )
$$
with $W$ a smooth function with bounded derivatives and such that, 
$$
\Big |G \Big ( \frac{v}{\sqrt{N}} \Big) \Big | \leq W \Big ( \frac{v}{\Delta^4 \sqrt{N}} \Big ).
$$
Write, 
$$
\varphi_{a,b, \kappa, \eta, u_1}(v; h) := \frac{\kappa \overline{4 a^2 (v + h a^2 u_1)}}{b^2} - \frac{\eta \overline{4 a (v + h a^2 u_1)^2}}{b}.
$$
Expanding the square we find that, 
$$
\mathcal{S} \ll \frac{\Delta^4 \sqrt{N}}{H} + \frac{1}{H^2} \sum_{\substack{|h_1|, |h_2| \leq H \\ h_1 \neq h_2}} \Big | \mathcal{S}_{h_1, h_2} \Big |
$$
where $\mathcal{S}_{h_1, h_2}$ is defined as
$$
\sum_{\substack{v \in \mathbb{Z} \\ (v + h_1 a^2 u_1, b) = 1 \\ (v + h_2 a^2 u_1, b) = 1}} e \Big ( \varphi_{a,b,\kappa,\eta, u_1}(v; h_1) - \varphi_{a,b,\kappa,\eta, u_1}(v; h_2) \Big ) W \Big ( \frac{v}{\Delta^4 \sqrt{N}} \Big ).
$$
By Poisson summation, 
\begin{equation} \label{eq:poisson201}
\mathcal{S}_{h_1, h_2} = \frac{\Delta^4 \sqrt{N}}{b^2}\sum_{\ell \in \mathbb{Z}} S(h_1, h_2, \ell) \widehat{W} \Big ( \frac{\ell \Delta^4 \sqrt{N}}{b^2} \Big )
\end{equation}
where $S(h_1, h_2, \ell)$ is defined as
$$
\sum_{\substack{x \mmod{b^2} \\ ( x + h_1 a^2 u_1, b) = 1 \\ (x + h_2 a^2 u_1, b) = 1 }} e \Big (  \frac{\kappa \overline{4 a^2 (x + h_1 a^2 u_1)} - \kappa \overline{4 a^2 (x + h_2 a^2 u_1)}}{b^2} -  \frac{\eta \overline{4 a (x + h_1 a^2 u_1)^2} - \eta \overline{4 a (x + h_2 a^2 u_1)^2}}{b} - \frac{\ell x}{b^2} \Big ).
$$
To evaluate $S(h_1, h_2, \ell)$ we write $x = \alpha + \beta b$ with $\alpha, \beta \mmod{b}$. Then the $\beta$ variable runs freely, and
$$
\overline{x + h_i a^2 u_1} \equiv \overline{\alpha + h_i a^2 u_1} - b (\overline{\alpha + h_i a^2 u_1})^2 \beta \mmod{b^2}.
$$
Therefore upon summing over $\beta$ we obtain a factor of $b$ provided that
$$
\kappa \overline{4 a^2} \Big ( (\overline{\alpha + h_1 a^2 u_1})^2  - (\overline{\alpha + h_2 a^2 u_1})^2 \Big ) \equiv \ell \mmod{b}.
$$
This is equivalent to
$$
\kappa \Big ( (\alpha + h_2 a^2 u_1)^2 - (\alpha + h_1 a^2  u_1)^2 \Big ) 
\equiv 
4 a^2 \ell (\alpha + h_1 a^2 u_1)^2 (\alpha + h_2 a^2 u_1)^2 \mmod{b}.
$$
We now count the number of solutions to this equation in $\alpha$ prime by prime, incidentally we recall that $b$ is square-free and that $(4 a u_1, b) = 1$. Let $p | b$. 
If $p \nmid \ell$ then there are at most $4$ solutions to the above equation modulo $p$, since $(4a,b) = 1$, so we can transform this equation into a monic polynomial of degree $4$. If $p | \ell$ then there is exactly one solution unless $p | \kappa (h_1 - h_2)$. Altogether this gives us,
$$
S(h_1, h_2, \ell) \ll 4^{\omega(b)} \cdot b \cdot (\ell \kappa (h_1 - h_2), b) \ll_{\varepsilon} b^{1 + \varepsilon} |\kappa|  H \cdot (\ell, b).
$$ 
We also notice that the $\ell = 0$ term admits a sharper bound of the form
$$
\ll b \cdot (\kappa (h_1 - h_2), b) \ll b |\kappa| H .
$$
As $b \asymp Q^{1 - 1/1000}$ and $H = Q^{1/2000}$, plugging these bounds into \eqref{eq:poisson201}
shows that \eqref{eq:poisson201} is 
$$
\ll \frac{\Delta^4 \sqrt{N}}{b^2} \cdot b^{1 + \varepsilon} |\kappa| H \cdot \Big ( 1 + \sum_{0 < |\ell| < Q^{1 - 1/500 + \varepsilon}} (b,\ell) \Big ) \ll_{\varepsilon} N^{\varepsilon} \cdot \frac{\sqrt{N}}{Q^{1 - 1/1000}} \cdot H  \cdot Q^{1 - 1/500} \ll N^{1/2 - 1 / 4000 + \varepsilon}.
$$
Therefore $\mathcal{S} \ll N^{1/2 - 1 / 4000 + \varepsilon}$ and therefore \eqref{eq:mainexp} is $\ll_{\Psi, V, \varepsilon} N^{1/2 - 1 / 8000 + \varepsilon}$, as needed. 

\section{Proof of Proposition \ref{prop:smoothingprop}}

Let $\mathcal{B}_{(1,2)}$ denote the region $(1 - \eta, 1) \cup (2, 2 + \eta)$ and let $\mathcal{B}_{(\eta,s)}$ denote the region $(0, \eta) \cup (s, s + \eta)$. The condition $N \| \sqrt{n} - x \| \in (0,s)$ is equivalent to $N (\sqrt{n} \mmod{1} - x) \in (0,s)$ except for $x$ in a neighborhood of size $\leq 10 s / N \leq s \eta$ around $x = 1$ and $x = 0$. Let $E_{n, s}(x)$ denote the event $N \| \sqrt{n} - x \| \in (0,s).$
We then have, 
$$
\Big | \mathbf{1} \Big ( \sum_{N \leq n < 2N} \mathbf{1}_{E_{n, s}(x)} = 0 \Big )  - \mathbf{1} \Big ( \sum_{n \in \mathbb{Z}} \mathbf{1}_{E_{n,s}(x)} V \Big ( \frac{n}{N} \Big ) = 0 \Big ) \Big | \leq \sum_{n \in N \mathcal{B}_{(1,2)}} \mathbf{1}_{E_{n,s}(x)},
$$
and also 
\begin{align*}
\Big | \mathbf{1} \Big ( \sum_{n \in \mathbb{Z}} \mathbf{1}_{E_{n,s}(x)} V \Big ( \frac{n}{N} \Big ) = 0 \Big ) - \mathbf{1} \Big ( & \sum_{n \in \mathbb{Z}} \Phi \Big ( N \| \sqrt{n} - x \| \Big ) V \Big ( \frac{n}{N} \Big ) = 0 \Big ) \Big | \\ & \leq \sum_{n \in \mathbb{Z}} \mathbf{1} \Big ( N \| \sqrt{n} - x \| \in \mathcal{B}_{(\eta,s)} \Big ) V \Big ( \frac{n}{N} \Big ).
\end{align*}
Therefore, 
\begin{align*}
\Big | \mathbf{1} \Big ( \sum_{N \leq n < 2N} \mathbf{1}_{E_{n, s}(x)} = 0 \Big )  - \mathbf{1} & \Big ( \sum_{n \in \mathbb{Z}} \Phi \Big ( N \| \sqrt{n} - x \| \Big ) V \Big ( \frac{n}{N} \Big ) = 0 \Big )  \Big | \\ \leq &  \sum_{n \in N \mathcal{B}_{(1,2)}} \mathbf{1}_{E_{n,s}(x)} + \sum_{n \in \mathbb{Z}} \mathbf{1} \Big ( N \| \sqrt{n} - x \| \in \mathcal{B}_{(\eta,s)} \Big ) V \Big ( \frac{n}{N} \Big )
\end{align*}
Integrating over $x$ and excluding neighborhoods of size $\leq s \eta$ around $x = 1$ and $x = 0$ we get 
$$
|\mathcal{V}(s, N; \mu) - \mathcal{V}_{\Phi, V}(N; \mu)| \leq 2 s \eta + 2 \eta + 2 \eta \cdot (1 + 2 \eta) \leq 8 (1 + s) \eta.
$$

\section{Proof of Proposition \ref{prop2}}

Consider, 
$$
\mathcal{V}_{\Phi,V} (N; \mu_{\Delta N, \mathcal{Q}_{\Delta, N}}) = \int_{0}^{1} \mathbf{1} \Big ( R_{\Phi, V}(x; N) = 0\Big )  d \mu_{\Delta N, \mathcal{Q}_{\Delta, N}}(x). 
$$
We now abbreviate $R = R_{\Phi, V}(x; N)$ and $\widetilde{R} := R_{\Phi, V}(x; N)$. 
We recall that, $0 \leq R_{\Phi, V}(x; N) \ll_{\Phi, V} \Delta^{10}$ for $x$ in the support of the measure $\mu_{N, \mathcal{Q}_{\Delta, N}}$. Therefore upon setting $\kappa = 1000 \log \Delta$, we get for $x$ in the support of $\mu_{N, \mathcal{Q}_{\Delta, N}}$,
\begin{align*}
\mathbf{1} ( R = 0 ) & = \exp( - \kappa R ) + O_{\Phi, V}(\Delta^{-10}) + O(\mathbf{1}_{0 < R < 1/100}) \\ & = \exp( - \kappa \widehat{\Phi}(0) \widehat{V}(0)) \exp( - 2 \kappa \widetilde{R}) + O_{\Phi, V} (\Delta^{-10}) + O(\mathbf{1}_{0 < R < 1/100}) \\ & = \exp( - \kappa \widehat{\Phi}(0) \widehat{V}(0)) \Big ( \sum_{0 \leq \ell \leq \Delta^{100}} \frac{(-2 \kappa)^{\ell}}{\ell!} \cdot \widetilde{R}^{\ell} \Big ) + O_{\Phi, V}(\Delta^{-10}) + O(\mathbf{1}_{0 < R < 1/100})
%\sum_{0 \leq \ell \leq \Delta^{100}} \frac{(-1)^{\ell} R^{\ell}}{\ell!} + O(\Delta^{-10} + \mathbf{1}_{0 < R < 1/100}) .
\end{align*}
Integrating over $x \in (0,1)$ with respect to the measure $\mu_{N, \mathcal{Q}_{\Delta, N}}$ and using Proposition \ref{prop4} we conclude that, 
\begin{align*}
\liminf_{N \rightarrow \infty} \ \mathcal{V}_{\Phi,V} (N; \mu_{N, \mathcal{Q}_{\Delta, N}}) & - \limsup_{N \rightarrow \infty} \ \mathcal{V}_{\Phi,V} (N; \mu_{N, \mathcal{Q}_{\Delta, N}}) \\ & = O_{\Phi, V}(\Delta^{-10}) + O \Big (\int_{0}^{1} \mathbf{1}_{0 < R_{\Phi, V}(x; N) < 1 / 100} \cdot dx \Big ).
\end{align*}
It remains to notice that 
$$\mathbf{1} \Big ( R_{\Phi, V}(x; N) \in (0, \tfrac{1}{100}) \Big ) \leq \sum_{n} \mathbf{1} \Big ( N \| \sqrt{n} - x \| \in (0,\eta) \cup (s, s + \eta) \Big ) V \Big ( \frac{n}{N} \Big )$$
since $\Phi(x) \in \{0,1\}$ for all non-negative $x$ except $x \in (0, \eta) \cup (s, s + \eta)$. 
Integrating over $x \in (0,1)$ then yields the required bound, 
$$
\int_{0}^{1} \mathbf{1}_{0 < R_{\Phi, V}(x; N) < 1/100} \cdot dx \ll \eta
$$
with an absolute constant in the $\ll$ since $0 \leq V \leq 1$ and $V$ is compactly supported in $(1/2, 3)$.

\section{Proof of the Elkies-McMullen theorem}\label{se:proofall}

By Proposition \ref{prop0} the existence of the gap distribution follows from the existence of a limiting behavior for the void statistic. Thus it is enough to show that
$$
\lim_{N \rightarrow \infty} \mathcal{V}(s; N)
$$
exists for each given $s > 0$.  By Proposition \ref{prop1} we have,
$$
\mathcal{V}(s,N)
 = \mathcal{V}(s,N; \mu_{N, \mathcal{Q}_{\Delta, N}}) + O \Big ( \frac{1}{\Delta^{1/2}} \Big ). 
 $$
 We select smooth functions $0 \leq \Phi, V \leq 1$ compactly supported in respectively $(0,s)$ and $(1 - \eta,2 + \eta)$ and
 equal to one in respectively $(\eta,s)$ and $(1, 2)$.
 By Proposition \ref{prop:smoothingprop} we have,
 $$
 \mathcal{V}_{\Delta}(s, N) = \mathcal{V}_{\Phi, V, \Delta}(N) + O((1 + s) \eta)
$$
for $N > 100 / \eta$, 
 while by Proposition \ref{prop2},
$$ 
\limsup_{N \rightarrow \infty} \mathcal{V}_{\Phi, V, \Delta}(N)
 - \liminf_{N \rightarrow \infty} \mathcal{V}_{\Phi, V, \Delta}(N) 
= O_{\Phi, V}(\Delta^{-10}) + O(\eta) .
$$
Combining all these together, it follows that,
$$
\limsup_{N \rightarrow \infty} \mathcal{V}(s; N) 
- \liminf_{N \rightarrow \infty} \mathcal{V}(s; N) = 
O_{\Phi, V}(\Delta^{-1/2}) + O((1 + s) \eta).
$$
We let $\Delta \rightarrow \infty$ and then $\eta \rightarrow 0$, and the claim follows\footnote{Note that we cannot first let $\eta \rightarrow 0$ and then $\Delta \rightarrow \infty$ because letting $\eta \rightarrow 0$ first might make the implicit constant in $O_{\Phi, V}(\Delta^{-10})$ grow out of bounds as both $\Phi$ and $V$ depend on $\eta$}. 
\section{Appendix}

We show here that the existence of a the gap distribution of a sequence follows from the existence of the void statistic. 
Specifically that Proposition \ref{prop0} implies the main theorem. 
This is basically already shown in \cite{Marklof}. Here we recall his proof. First we note that the sequence of measures
$$
F_{N} := \frac{1}{N} \sum_{N \leq n < 2N} \delta  ( N (s_{n + 1, N} - s_{n, N} ) )
$$
is tight; this follows simply because the first moment is uniformly bounded, 
$$
\frac{1}{N} \sum_{N \leq n < 2N} (N (s_{n + 1, N} - s_{n, N})) \leq 1.
$$
As a result the sequence of measures $F_{N}$ is relatively compact. Therefore for any subsequence $N_k$ we can find a convergent subsequence $\{ N_k' \} \subset \{ N_k \}$ such that $F_{N_k'}$ converges weakly to a limit $P^{\star}$. We aim to show that all limits $P^{\star}$ coincide, as a result the entire sequence of measures $F_{N}$ is weakly convergent.

Suppose to the contrary that we can find two subsequences $\{M_k\}$ and $\{M_k'\}$ such that $F_{M_k}$ converges weakly to $P_1$ while $F_{M_k'}$ converges weakly to $P_2$. Since $P_1$ and $P_2$ are assumed not to coincide we can find a smooth compactly supported function $f$ such that,
$$
\int_{0}^{\infty} f(x) \,d F_{M_k} (x)
\rightarrow \int_0^{\infty} f(x) \,d P_1(x) 
\neq \int_{0}^{\infty} f(x) \,d P_{2}(x) 
\leftarrow \int_{0}^{\infty} f(x) \,d F_{M_k'}(x).
%\ \int_{0}^{\infty} f d P_1 \neq \int_{0}^{\infty} f d P_2
$$
Since the set of functions spanned 
by linear combinations of
$g_{L}(x) := \max(x, L)$ is dense in
$[0, \infty]$ there exists an $L$ such that
$$
\int_{0}^{\infty} g_{L}(x) \, d P_1(x) 
\neq \int_{0}^{\infty} g_{L}(x) \, d P_2(x). 
$$
By the computation in \cite{Marklof} we also have, 
$$
\frac{1}{M_k} \int_{0}^{1} 
\mathbf{1} ( \mathcal{N}(x; L, M_k) = 0 ) \, dx 
\rightarrow \int_{0}^{\infty} g_{L}(x) d P_1(x)
$$
and 
$$
\frac{1}{M_k'} \int_{0}^{1} \mathbf{1} 
( \mathcal{N}(x; L, M_k') = 0 ) \, dx
\rightarrow \int_{0}^{\infty} g_{L}(x) d P_2(x).
$$
However Proposition \ref{prop0} shows that the limit of the void statistic exists, so we get,
$$
\int_{0}^{\infty} g_{L}(x) \, d P_1(x) 
= \int_{0}^{\infty} g_{L}(x) \, d P_2(x),
$$
a contradiction. 

\subsection*{Acknowledgements}
M.R. was supported by NSF grant DMS-1902063.
NT was supported by the Austrian Science Fund (FWF): 
project number J 4464-N. 
While working on this paper, he visited
the California Institute of Technology 
and Graz University of Technology,
whose hospitality he acknowledges.
\bibliography{em}
\bibliographystyle{plain}

\end{document}